\documentclass[11pt]{amsart}
\usepackage{amsmath,amssymb,bm}
\usepackage{verbatim,a4wide} 
\newtheorem{theorem}{Theorem}
\newtheorem{lemma}{Lemma}[section]
\newtheorem{proposition}[lemma]{Proposition}
\newtheorem{corollary}[lemma]{Corollary}

\newtheorem{remark}[lemma]{Remark}
\newtheorem{definition}[lemma]{Definition}
\numberwithin{equation}{section}
\newcommand{\R}{\mathbb{R}}
\newcommand{\C}{\mathbb{C}}
\newcommand{\N}{\mathbb{N}}
\newcommand{\Z}{\mathbb{Z}}
\newcommand{\T}{\mathbb{T}}

\begin{document}
\title
[Global solutions of super-critical regularity]
{Probabilistic well-posedness for the cubic wave equation}
\author{Nicolas Burq and Nikolay Tzvetkov}
\address{Laboratoire de Math\'ematiques, University Paris-Sud 11, F-91405, UMR 8628 du CNRS et  Universit\'e de 
Cergy-Pontoise,  Cergy-Pontoise, F-95000,UMR du CNRS 8088, et Institut Universitaire de france. }
\thanks{The first author is supported by ANR project EQUADISP, while the second author is supported by ERC project Dispeq}
\begin{abstract}
The purpose of this article is to introduce for dispersive partial differential equations with random initial data, the notion of well-posedness (in the Hadamard-probabilistic sense). We restrict the study to one of the simplest examples of such equations:  the periodic cubic semi-linear wave equation. Our contributions in this work are twofold: first we break the algebraic rigidity involved in previous works and allow much more general randomizations (general infinite product measures v.s. Gibbs measures), and second, we show that the flow that we are able to construct enjoys very nice dynamical properties, including a new notion of probabilistic continuity.
\end{abstract}
\maketitle
\section{Introduction}
In \cite{BT1} , we developed a general method for obtaining local existence and uniqueness of semi-linear wave equations with data of super-critical regularity.
In addition, in \cite{BT2} we gave a very particular example (based on invariant measures considerations) 
of global existence with data of supercritical regularity. 
Our goal here is to make a significant extension 
of \cite{BT1,BT2} by presenting a quite general scheme to get global well-posedness for semi-linear dispersive equations
with data of super-critical regularity.  
We also propose a natural notion of Hadamard well-posedness in this setting.
We decided to restrict our attention to a very simple example. A further development of the ideas we introduce here
will appear in a forthcoming work.

Let $(M,g)$ be a $3d$ boundaryless Riemannian manifold with associated Laplace-Beltrami operator $\Delta_{g}$. Consider the cubic defocusing wave equation
\begin{equation}\label{NLW}
\begin{gathered}
(\partial_t^2-\Delta_{g})u+u^3=0,\quad u:\R \times M\rightarrow \R,
\\ 
u|_{t=0}=u_0,\,\, \partial_t u|_{t=0}=u_1,\quad\quad (u_0,u_1)\in H^s(M)\times H^{s-1}(M)\equiv {\mathcal H}^s(M),
 \end{gathered}
\end{equation}
where $H^s(M)$ denotes the classical Sobolev spaces on $M$.
By using simple scaling considerations one obtains that $s=1/2$ is the critical Sobolev regularity associated to \eqref{NLW}.
It turns out that this regularity is the border line of the deterministic theory, in the sense of local well-posedness in the Hadamard sense (existence,
uniqueness and continuous dependence on the data).
More precisely, we have the following statement.
\begin{theorem}\label{th1}
The Cauchy problem \eqref{NLW} is locally well-posed for data in ${\mathcal H}^s$, $s\geq 1/2$ (and even globally for $s\geq 3/4$).
In the opposite direction, for $s\in(0,1/2)$, the Cauchy problem \eqref{NLW} is not locally well-posed in ${\mathcal H}^s$. For instance
one can contradict the continuous dependence by showing that there exists a sequence $(u_n)$ of global smooth solutions of \eqref{NLW} such that
$$
\lim_{n\rightarrow\infty}\|(u_n(0),\partial_t u_n(0))\|_{{\mathcal H}^s}=0
$$
and 
$$
\lim_{n\rightarrow\infty}\|(u_n(t),\partial_t u_n(t))\|_{L^\infty([0,T];{\mathcal H}^s)}=\infty,\quad \forall\, T>0.
$$
\end{theorem}
Moreover, one can also contradict the existence by showing that there exists an initial datum $(u_0,u_1)\in {\mathcal H}^s$ 
such that for every $T>0$ the problem \eqref{NLW} has no solution in $L^\infty([0,T];{\mathcal H}^s)$, 
if we suppose that in addition the flow satisfies a suitable finite speed of propagation.
Such a property is natural in the context of wave equations.

One may wish to compare the result of Theorem~\ref{th1} with the classical Hadamard counterexample in the context of the
Laplace equation
\begin{equation}\label{laplace}
(\partial_t^2+\partial_x^2)v=0,\quad v\,:\, \R_t\times S^1_{x}\longrightarrow \R.
\end{equation}
The equation \eqref{laplace} has the explicit solution
$$
v_{n}(t,x)=e^{-\sqrt{n}}{\rm sh}(nt)\cos(nx).
$$
Then for every $s$, $v_n$ satisfies
$$
\|(v_n(0),\partial_t v_n(0))\|_{{\mathcal H}^s(S^1)}\lesssim  e^{-\sqrt{n}}n^{s}\longrightarrow 0,
$$
as $n$ tends to $+\infty$ but for $t\neq 0$,
$$
\|(v_n(t),\partial_t v_n(t))\|_{{\mathcal H}^s(S^1)}\gtrsim  e^{n|t|}\,e^{-\sqrt{n}}n^{s}\longrightarrow +\infty,
$$
as $n$ tends to $+\infty$.
Consequently \eqref{laplace} in not well-posed in ${\mathcal H}^s$ for every $s\in\R$ which is the analogue of the ${\mathcal H}^s$, $s<1/2$ result of
Theorem~\ref{th1}.  On the other hand \eqref{laplace} is well-posed in analytic spaces which is the analogue of the $s>1/2$ result in Theorem~\ref{th1}.
Let us also observe that one may show the ill-posedness of \eqref{laplace} in Sobolev spaces by an indirect argument based on elliptic regularity.
We are not aware of a similar indirect argument in the context of the wave equation \eqref{NLW} for $s<1/2$.  

The well-posedness part of the Theorem~\ref{th1} can be proved as in the works by Ginibre-Velo~\cite{GiVe} and Lindblad-Sogge~\cite{LS}, 
by invoking the Strichartz estimates for the wave equation on a riemannian manifold
due to Kapitanskii \cite{K}. For $s>1/2$ the well-posedness holds in a stronger sense since the time existence can be chosen the same for all data in a fixed bounded set of ${\mathcal H}^s$ and moreover the flow map is uniformly continuous on bounded sets of ${\mathcal H}^s$.
In the case $s=1/2$ the situation is more delicate since the existence time depends in a more subtle way on
the data. 
The global well-posedness part of Theorem~\ref{th1} can be obtained (following ideas by Bourgain~\cite{Bo}) by adapting the proofs of Kenig-Ponce-Vega~\cite{KPV}, Gallagher-Planchon~\cite{GaPl} and Bahouri-Chemin~\cite{BaCH} to the compact setting.
We also refer to the works by Roy~\cite{Roy} for further investigations in the direction of deterministic global well-posedness for \eqref{NLW} with rough data.
The ill-posedness statement of Theorem~\ref{th1} is proved in our previous article~\cite{BT1}, by using the approaches of Christ-Colliander-Tao~\cite{CCT} and 
Lebeau~\cite{L}. 
\par
One may however ask whether some sort of well-posedness for \eqref{NLW} survives for $s<1/2$. In \cite{BT1} we have shown that the
answer is positive, at least locally in time, if one accepts to randomize the initial data. Moreover, the method of \cite{BT1} works for a quite general
class of randomizations. As already mentioned the approach of \cite{BT2} to get global in time results is restricted only to very particular randomizations. 
More precisely, it is based on a global control on the flow given by an invariant Gibbs measure (see also \cite{Bo}). In \cite{BT1}, Remark~1.5, we asked
whether the globalization argument can be performed by using other global controls on the flow such as conservations laws. In the present work we 
give a positive answer to this question. 
\par 
Let us now describe the initial data randomization we use. We suppose that $M=\T^3$ with the flat metric. Starting from 
$(u_0,u_1)\in {\mathcal H}^s$ given by their Fourier series 
$$ 
u_{j}(x)=a_{j}+\sum_{n\in\Z^3_{\star}}\Big(b_{n,j}\,\cos(n\cdot x)+c_{n,j}\sin(n\cdot x)\Big),\qquad \quad j=0,1,
\quad \Z^3_{\star}=\Z^3\backslash\{0\}
$$
we define $u_{j}^\omega$ by
\begin{equation}\label{coord}
u_{j}^\omega(x)=\alpha_{j}(\omega)a_{j}+\sum_{n\in\Z^3_{\star}}\Big(\beta_{n,j}(\omega)b_{n,j}\,\cos(n\cdot x)+\gamma_{n,j}(\omega)c_{n,j}\sin(n\cdot x)\Big),
\end{equation}
where $(\alpha_{j}(\omega),\beta_{n,j}(\omega),\gamma_{n,j}(\omega))$, $n\in\Z^3_{\star}$, $j=0,1$ 
is a sequence of real random variables on a probability space $(\Omega,p,{\mathcal F})$.
We assume that the random variables $(\alpha_j,\beta_{n,j},\gamma_{n,j})_{n\in\Z^3_{\star},j=0,1}$ 
are independent identically distributed real random variables with a joint distribution $\theta$ satisfying 
\begin{equation}\label{subgauss}
\exists\, c>0,\quad \forall\,\gamma\in\R,\quad
\Big|\int_{-\infty}^{\infty}e^{\gamma x}d\theta(x)\Big|
\leq e^{c\gamma^2}
\end{equation}
(under the assumption \eqref{subgauss} the random variables are necessarily of mean zero).
Typical examples of random variables satisfying \eqref{subgauss} are the standard gaussians, i.e. $d\theta(x)=(2\pi)^{-1/2}\exp(-x^2/2)dx$ 
(with an identity in \eqref{subgauss}) or any family of random variables having compactly supported distriution function $\theta$, e.g. the Bernoulli variables $d\theta(x)=\frac{1}{2}(\delta_{-1}+\delta_{1})$. An advantage of the Bernoulli randomization is that it keeps the ${\mathcal H}^s$ norm of the original function.
The gaussian randomization has the advantage to "generate" a dense set in ${\mathcal H}^s$ via the map 
\begin{equation}\label{eq.proba}
 \omega \in \Omega \longmapsto (u_0^\omega, u_1^\omega)\in{\mathcal H}^s
 \end{equation}
for many $(u_0,u_1)\in {\mathcal H}^s$.  Notice finally that we could relax the "identical distribution" assumption provided~\eqref{subgauss} is uniformly satisfied by the family of random variables.

\begin{definition} For fixed $(u_0, u_1) \in \mathcal{H}^s$, the map \eqref{eq.proba} is a measurable map from $(\Omega,{\mathcal F})$ to ${\mathcal H}^0$ endowed with the Borel sigma algebra 
since the partial sums from a Cauchy sequence
in $L^2(\Omega;{\mathcal H}^0)$. Thus \eqref{eq.proba} endows the space ${\mathcal H}^0(\T^3)$ with a probability measure which is direct image of $p$. Let us denote this measure by $\mu_{(u_0, u_1)}$. Then
$$\forall\, A \subset \mathcal{H}^0,\,\, 
\mu_{(u_0, u_1)} (A)= p ( \omega\in \Omega\,:\, (u_0^\omega, u_1^\omega) \in A).
$$
Denote by ${\mathcal M}^s$ the set of measures obtained following this construction and
$${\mathcal M}^s= \bigcup_{(u_0, u_1) \in \mathcal{H}^s} \{ \mu_{(u_0, u_1)}\}\,.
$$
\end{definition}
Let us recall some basic properties of these measures (see~\cite{BT1}).
\begin{proposition} 
For any $(u_0, u_1) \in \mathcal{H}^s$, the measure $ \mu_{(u_0, u_1)}$ is supported by $\mathcal{H}^s$. Furthermore, for any $s'>s$, if $(u_0, u_1) \notin \mathcal{H}^{s'}$, then $ \mu_{(u_0, u_1)}( \mathcal{H}^{s'})=0$. In other words,  the randomization \eqref{eq.proba} does not regularize in the scale
of the $L^2$-based Sobolev spaces (this fact is obvious for the Bernoulli randomization). Finally,
If $(u_0,u_1)$ have all their Fourier coefficients different from zero and if the measure $\theta$ charges all open sets of $\R$ then the support of $\mu$ is  
${\mathcal H^s}$ (recall that the support of $\mu$ is the complementary of the largest open set $U\subset {\mathcal H}^s$ such that $\mu(U)=0$).
\end{proposition}
As mentioned above, for fixed $(u_0,u_1)$ the measure $\mu_{(u_0,u_1)}$ depends heavily on the choice of the random variables
$(\alpha_{j}(\omega),\beta_{n,j}(\omega),\gamma_{n,j}(\omega))$. On the other hand for a fixed randomisation 
$(\alpha_{j}(\omega),\beta_{n,j}(\omega),\gamma_{n,j}(\omega))$ the measure $\mu_{(u_0,u_1)}$ depends largely on the choice of $(u_0, u_1)$. 
 For instance, let us consider a gaussian randomisation , i.e. we suppose that $(\alpha_{j}(\omega),\beta_{n,j}(\omega),\gamma_{n,j}(\omega))$
 are independent centered gaussian random variables.
 Then if $(u_0,u_1)$ and $(\widetilde{u}_0,\widetilde{u}_1)$ are given by the Fourier expansions  
 $$ 
u_{j}(x)=a_{j}+\sum_{n\in\Z^3_{\star}}\Big(b_{n,j}\,\cos(n\cdot x)+c_{n,j}\sin(n\cdot x)\Big),\qquad \quad j=0,1
$$
and
$$ 
\widetilde{u}_{j}(x)=\widetilde{a}_{j}+\sum_{n\in\Z^3_{\star}}\Big(\widetilde{b}_{n,j}\,\cos(n\cdot x)+\widetilde{c}_{n,j}\sin(n\cdot x)\Big),\qquad \quad j=0,1
$$
 then, following Kakutani~\cite{Ka}, it is possible to prove that the associated measures $\mu_{(u_0,u_1)}$ and $\mu_{(\widetilde{u}_0,\widetilde{u}_1)}$ are mutually singular if
 \begin{equation}\label{sing}
\sum_{n}
\Big| \frac{\widetilde{b}_{n,j}}{b_{n,j}}-1\Big|^2+\Big| \frac{\widetilde{c}_{n,j}}{c_{n,j}}-1\Big|^2=+\infty .
 \end{equation}
 In other words, if \eqref{sing} is satisfied then there exists a set $A$ such that $\mu_{(u_0,u_1)}(A)=1$ and 
 $\mu_{(\widetilde{u}_0,\widetilde{u}_1)}(A)=0$.
 On the other hand if \eqref{sing} is not satisfied and the all the coefficients are non zero (or vanish pairwise simultaneously) then we have that $\mu_{(u_0,u_1)}\ll \mu_{(\widetilde{u}_0,\widetilde{u}_1)}\ll \mu_{({u}_0,{u}_1)}$ We refer to the Appendix for more precise statements concerning the dependence of $\mu_{(u_0,u_1)}$ on $(u_0,u_1)$ in the case of a gaussian randomisation.
 
We can now state our first result.
\begin{theorem}\label{main}
Let $M=\T^3$ with the flat metric and let us fix $\mu\in {\mathcal M}^s$, $0\leq s <1$. 
Then, there exists a full $\mu$ measure set $\Sigma \subset  {\mathcal H}^s(\T^3)$ 
such that for every $(v_0, v_1)\in \Sigma$, there exists a unique global solution $v$ of
of the non linear wave equation
 \begin{equation}\label{valna}
 (\partial_t^2-\Delta_{\T^3})v+v^3=0,\quad (v(0),\partial_t v(0))=(v_0,v_1)
 \end{equation}
satisfying
 $$
(v(t),\partial_t v(t)) \in \big(S(t)(v_0, v_1),\partial_t S(t)(v_0, v_1)\big)+ C(\R_t; H^1(\T^3) \times L^2(\T^3))
 $$ 
 ($S(t)$ denotes the free evolution defined by \eqref{free} below).
Furthermore, if we denote by 
$$\Phi(t) (v_0, v_1)\equiv (v(t),\partial_t v(t))$$ 
the flow thus defined, the set $\Sigma$ is invariant by the map $\Phi(t)$, namely
$$
\Phi(t)(\Sigma)=\Sigma,\qquad \forall\, t\in \R.
$$
Finally, for any $\varepsilon>0$ there exist $C, \delta >0$ such that for $\mu$ almost every $(v_0, v_1)\in \mathcal{H}^s( \T^3)$, there exists $M>0$ such that the global solution to~\eqref{valna} previously constructed satisfies 
$$
v(t)= S(t) \Pi^0(v_0, v_1)+ w(t),
$$
($\Pi_0$ is the orthogonal projector on constants), with
\begin{equation}\label{eq.largetime1} 
\|(w(t), \partial_t w(t)) \|_{\mathcal{H}^1(\T^3)} \leq 
\begin{cases} C (M+ |t|)^{\frac {1-s} s + \varepsilon} &\text{ if  $s>0$},\\
 C e^{C(t+M)^2} &\text{ if  $s=0$},
\end{cases}
 \end{equation}
and 
\begin{equation*}
\mu ((v_0,v_1)\in {\mathcal H}^s\,:\,M>\lambda) \leq  C e^{-\lambda^\delta}. 
\end{equation*}
\end{theorem}
Having established a large time (unique) dynamics on an invariant set of full measure on ${\mathcal H}^s(\T^3)$, 
there are a few very natural questions to address, and the very first one is the continuity of the flow. 
 Let us  recall that for any event $B$ (of non null probability) the conditionned probability $\mathcal{P}( \cdot \vert B)$ is the natural probability measure supported by $B$, defined by 
$$ \mathcal{P} ( A\vert B) = \frac{ \mathcal{P} (A\cap B) } { \mathcal{P}( B)}
$$ 
Notice (see below), that the sequences constructed following the approach by Lebeau and Christ-Colliander-Tao give an obstruction to the (deterministic) continuity of our flow. However, we are able to prove that it is still continuous {\em in probability} and consequently the super-critical Cauchy problem~\eqref{NLW}  
is well globally posed in the following {\em Hadamard-probabilistic} sense
\begin{theorem}\label{th_continuity}
Let us fix  $s\in (0,1)$,  let $A>0$ and let $B_A\equiv (V\in {\mathcal H}^s : \|V\|_{{\mathcal H}^s}\leq A)$  
be the closed ball of radius $A$ centered at the origin of ${\mathcal H}^s$ and let $T>0$. 
Let  $\mu\in {\mathcal M}^s$ and suppose that $\theta$ is symmetric.
Let $\Phi(t)$ be the flow of the cubic wave equations defined $\mu$ almost everywhere in Theorem~\ref{main}.
Then for $\varepsilon, \eta>0$,  we have the bound 
\begin{multline}\label{dimanche}
 \mu\otimes\mu\Big((V,V')\in {\mathcal H}^s\times {\mathcal H}^s\,:
  \| \Phi(t) (V) - \Phi(t) (V') \|_{X_T} >\varepsilon  \Bigm {\vert} 
  \\
  \| V-V'\|_{\mathcal{H}^s}< \eta \,\,{\rm and}\,\,( V,V')\in B_A\times B_A   \Big) \leq  g(\varepsilon,\eta),
\end{multline}
where
$
X_{T}\equiv (C ([0,T]; \mathcal{H}^s)\cap L^4([0,T]\times\T^3))\times C([0,T];H^{s-1})
$
and  $g(\varepsilon,\eta)$ is such that
$$
\lim_{\eta\rightarrow 0}g(\varepsilon,\eta)=0,\qquad \forall\,\varepsilon>0.
$$
Moreover, if  in addition we assume that the support of $\mu$ is the whole ${\mathcal H}^s$ then there exists 
$\varepsilon>0$ such that for every $\eta>0$ the left hand-side in \eqref{dimanche} is positive.
\end{theorem} 
In other words, as soon as $\eta \ll\varepsilon$, among the initial data which are $\eta$-close to each other, the probability of finding two for which the corresponding solutions to~\eqref{NLW} do not remain $\varepsilon$ close to each other, is very small ! 
The last part of the statement is saying that  the deterministic version of the uniform continuity property \eqref{dimanche} does not hold.
A crucial element in the proof is the ill-posedness result displayed in Theorem~\ref{th1}.
It is likely that Theorem~\ref{th_continuity} also holds for $s=0$, modulo some additional technicalities. 
\par
%
In a forthcoming work, we show that similar results could be obtained for general manifolds by modifying accordingly the randomization. 
\par
As mentioned in \cite{BT1} it would be interesting to develop similar results in the case of the nonlinear Schr\"odinger equation (NLS).
In this case there are at least two difficulties compared to the wave equation. The first one is that the smoothing in the nonlinearity is no longer present
in the  case of NLS. The second one is that the deterministic Cauchy theory in the case of NLS, posed on a compact domain is much more 
intricate compared to the nonlinear wave equation or the NLS in the euclidean space (see e.g \cite{Bo0,BGT1}). 
One can however show that in some cases one may at least control a.s. the first iteration at a super critical regularity (see the appendix of \cite{Tz}). Another approach based on subscribing the singular part of the nonlinearity is developed in \cite{Bo2,CO}.  Finally, let us mention that we obtained with Thomann a first step toward similar results for the non linear Schr\"odinger equation in~\cite{BTT}.
\par
The remaining part of this paper is organized as follows. We complete this introduction by introducing several notations.
In the following section, we give the global existence part of the proof of Theorem~\ref{main} for $s>0$.
Next, we construct an invariant set of full measure. Section~4 is devoted to the possible growth of Sobolev norms for $s>0$. 
We then consider the case $s=0$ is Section~5. Section~6 is devoted to the proof of Theorem~\ref{th_continuity}. Finally in an appendix, we collect the results
on random series used in the previous sections. We also prove statement giving a criterium for the orthogonality of two measures of ${\mathcal M}^s$.
\par
\noindent {\bf Notation.}  A probability measure $\theta$ on $\R$ is called symmetric if $\int_{\R}f(x)d\theta(x)=\int_{\R}f(-x)d\theta(x)$ for every $f\in L^1(d\theta)$.
A real random variable is called symmetric if its distribution is a symmetric measure on $\R$. 

\par
{\bf Acknowledgement.} We benefited from discussions with Jean-Pierre Kahane, Herbert Koch and Daniel Tataru. 
\section{Almost sure global well posedness for $s>0$}
Let us first recall the following local existence result.
\begin{proposition}\label{prop.local}
Consider the problem
\begin{equation}\label{model}
(\partial_t^2-\Delta_{\T^3})v+(f+v)^3=0\,.
\end{equation}
There exists a constant $C$ such that for every time interval $I=[a,b]$ of size $1$, 
every $\Lambda\geq 1$, every 
$
(v_0,v_1,f)\in H^1\times L^2\times L^3(I,L^6)
$
satisfying
$
\|v_0\|_{H^1}+\|v_1\|_{L^2}+\|f\|^3_{L^3(I,L^6)}\leq \Lambda
$
there exists a unique solution on the time interval $[a,a+C^{-1}\Lambda^{-2}]$ of \eqref{model} with initial data
$$
v(a,x)=v_0(x), \quad \partial_t v(a,x)=v_1(x)\,.
$$
Moreover the solution satisfies
$
\|(v,\partial_t v)\|_{L^\infty([a,a+\tau],H^1\times L^2)}\leq C\Lambda,
$
$(v,\partial_t v)$ is unique in the class $L^\infty([a,a+\tau],H^1\times L^2)$ and the dependence in time is continuous.
\end{proposition}
\begin{proof}
By translation invariance in time, we can suppose that $I=[0,1]$.
Define the free evolution $S(t)$ by
\begin{equation}\label{free}
S(t)(v_0,v_1)\equiv
\cos(t\sqrt{-\Delta})(v_0)+\frac{\sin(t\sqrt{-\Delta})}{\sqrt{-\Delta}}(v_1)
\end{equation}
with the natural convention concerning the zero Fourier mode.
Then we can rewrite the problem as
\begin{equation}\label{Duhamel}
v(t)=S(t)(v_0,v_1)-\int_{0}^t \frac{\sin(t\sqrt{-\Delta})}{\sqrt{-\Delta}}((f(\tau)+v(\tau))^3d\tau\,.
\end{equation}
Set
$$
\Phi_{v_0,v_1,f}(v)\equiv  S(t)(v_0,v_1)-\int_{0}^t \frac{\sin(t\sqrt{-\Delta})}{\sqrt{-\Delta}}((f(\tau)+v(\tau))^3d\tau.
$$
Then for $T\in (0,1]$, using the Sobolev embedding $H^1(\T^3)\subset L^6(\T^3)$, we get
\begin{multline*}
\|\Phi_{v_0,v_1,f}(v)\|_{L^\infty([0,T],H^1)} \leq  C(\|v_0\|_{H^1}+\|v_1\|_{L^2}+T\sup_{\tau\in[0,T]}\|f(\tau)+v(\tau)\|_{L^6}^3)
\\
\leq 
C(\|v_0\|_{H^1}+\|v_1\|_{L^2}+\sup_{\tau\in [0,T]}\|f(\tau)\|_{L^6}^3)+CT\|v\|^3_{L^\infty([0,T],H^1)}
\end{multline*}
It is now clear that for $T\approx \Lambda^{-2}$ the map $\Phi_{u_0,u_1,f}$ send the ball
$(v:\|v\|_{L^\infty([0,T],H^1)}\leq C\Lambda)$ onto. Moreover by a similar argument, we obtain that this map is a contraction on the
same ball. Thus we obtain the existence part and the bound on $v$ in $H^1$. The estimate of $\|\partial_t v\|_{L^2}$ follows by differentiating
in $t$ the Duhamel formula \eqref{Duhamel}. This completes the proof of Proposition~\ref{prop.local}.
\end{proof}
We can now deduce the global existence and uniqueness part in Theorem~\ref{main} in the case $s>0$.
\begin{proposition}\label{th.1}
Assume that $s>0$ and let us fix $\mu\in {\mathcal M}^s$.
Then, for $\mu$ almost every $(v_0, v_1) \in {\mathcal H}^s(\T^3)$, there exists a unique global solution 
$$
(v(t),\partial_t v(t)) \in (S(t)(v_0, v_1),\partial_t S(t)(v_0, v_1))+ C(\R; H^1(\T^3) \times L^2(\T^3))$$ 
of the non linear wave equation
\begin{equation*}
(\partial_t^2-\Delta_{\T^3})v+v^3=0
\end{equation*}
with initial data
$$
v(0,x)=v_0(x), \quad \partial_t v(0,x)=v_1(x)\,.
$$
\end{proposition}
\begin{proof}
We search $v$ under the form $v(t)= S(t) (v_0, v_1) +w(t)$. Then $w$ solves
\begin{equation}\label{eq.1}
(\partial_t^2-\Delta_{\T^3})w+(S(t)(v_0, v_1)+w)^3=0,\quad w\mid_{t=0} =0, \quad \partial_t w\mid_{t=0} = 0,
\end{equation}
From Corollary~\ref{random2}, if $\delta> 1+ \frac 1 p$, we know that $\mu$-almost surely
$$ \|\langle t\rangle^{-\delta}  S(t) (v_0, v_1)\|_{L^p( \mathbb{R}_t; W^{s,p}(\T^3))} <+\infty. $$
Taking $p$ large enough so that $\frac 3 p<s$ and consequently  $W^{s, p} (\T^3) \subset L^\infty(\T^3)$, we deduce that $\mu$-almost surely, 
\begin{equation}\label{eq.borne}
\begin{aligned}
g(t) &= \|S(t)(v_0, v_1)\|^3_{L^6( \T^3)}\in L^1_{\text{loc}} ( \R_t),
\\
f(t) &=  \|S(t)(v_0, v_1)\|_{L^\infty( \T^3)}\in L^1_{\text{loc}} ( \R_t).
\end{aligned}
\end{equation}
The local existence in Proposition~\ref{th.1} now follows from Proposition~\ref{prop.local} and the first estimate in~\eqref{eq.borne}. 
We also deduce from Proposition~\ref{prop.local}, 
that as long as the $H^1\times L^2$ norm of $(w, \partial_t w)$ remains bounded, the solution $w$ of~\eqref{eq.1} exists. Set
$$ 
{\mathcal E}(w(t)) = \frac 1 2 \int_{\T^3}\big( (\partial_t w)^2 + |\nabla_x w|^2 + \frac 1 2 w^4\big) dx\,.
$$
Using the equation solved by $w$, we now compute 
\begin{eqnarray*}
\frac{d} {dt} {\mathcal E}(w(t)) &=  &\int_{\T^3}\big( \partial_t w \partial_t^2 w + \nabla_x \partial _t w \cdot \nabla_x w + \partial_t w\, w ^3 \big)dx
\\
& =  &\int_{\T^3} \partial_t w \Bigl(\partial_t^2 w -\Delta w + w^3\Bigr) dx
\\
& =  &\int_{\T^3} \partial_t w \Bigl(w^3-  (S(t)(v_0, v_1)+ w)^3\Bigr) dx.
\end{eqnarray*}
Now, using the Cauchy-Schwarz and the H\"older inequalities, we can write
\begin{equation*}
\begin{aligned}
\frac{d} {dt} {\mathcal E}(w(t))  &\leq C \big({\mathcal E}(w(t))\big)^{1/2}  \|w^3-  (S(t)(v_0, v_1)+ w)^3\|_{L^2( \T^3)}\\
&\leq C \big({\mathcal E}(w(t))\big)^{1/2} 
\Bigl(\|S(t)(v_0, v_1)\|^3_{L^6( \T^3)} + \|S(t)(v_0, v_1)\|_{L^\infty( \T^3)} \|w^2\|_{L^2( \T^3)} \Bigr)\\
&\leq 
C \big({\mathcal E}(w(t))\big)^{1/2} 
 \Big(g(t) + f(t) \big({\mathcal E}(w(t))\big)^{1/2}  \Big)
\end{aligned}
\end{equation*}
and consequently, according to Gronwall inequality and~\eqref{eq.borne}, $w$ exists globally in time.
This completes the proof of Proposition~\ref{th1}.
\end{proof}
\section{Construction of an invariant set, $s>0$}
The construction of the previous section yields the global existence on a set of full $\mu$ measure but it does not exclude the possibility to have a dynamics 
sending for some $t\neq 0$ the set of full measure where the global existence holds to a set of small measure. 
Notice that in a similar discussion in \cite{BT2} the set of full measure where the global
existence holds is invariant by the dynamics by construction. 
Our purpose of this section is to establish a global dynamics on an invariant set of full measure in the context of the argument of the previous section.
\par
Define the sets
$$
\Theta\equiv
\big(
(v_0,v_1)\in {\mathcal H}^s\,:\,\|S(t)(v_0, v_1)\|^3_{L^6( \T^3)}\in L^1_{\text{loc}} ( \R_t)
,\,\,\, \|S(t)(v_0, v_1)\|_{L^\infty( \T^3)}\in L^1_{\text{loc}} ( \R_t)
\big)
$$
and
$
\Sigma\equiv\Theta+{\mathcal H}^1.
$
Then $\Sigma$ is of full $\mu$ measure for every $\mu\in {\mathcal H}^s$, since so is $\Theta$.
We have the following proposition.
\begin{proposition}\label{th.1.bis}
Assume that $s>0$ and let us fix $\mu\in {\mathcal M}^s$.
Then, for every $(v_0, v_1) \in \Sigma$, there exists a unique global solution 
$$
(v(t),\partial_t v(t)) \in (S(t)(v_0, v_1),\partial_t S(t)(v_0, v_1))+ C(\R; H^1(\T^3) \times L^2(\T^3))$$ 
of the non linear wave equation
\begin{equation}\label{shart}
(\partial_t^2-\Delta_{\T^3})v+v^3=0,\quad (v(0,x),\partial_{t} v(0,x))=(v_0(x),v_1(x))\, .
\end{equation}
Moreover for every $t\in\R$, 
$
(v(t),\partial_t v(t))\in\Sigma.
$
\end{proposition}
\begin{proof}
By assumption, we can write $(v_0,v_1)=(\tilde{v}_0,\tilde{v}_1)+(w_0,w_1)$ with
$(\tilde{v}_0,\tilde{v}_1)\in\Theta$ and $(w_0,w_1)\in {\mathcal H}^1$.
We search $v$ under the form $v(t)= S(t) (\tilde{v}_0, \tilde{v}_1) +w(t)$. Then $w$ solves
\begin{equation*}
(\partial_t^2-\Delta_{\T^3})w+(S(t)(\tilde{v}_0, \tilde{v}_1)+w)^3=0,\quad w\mid_{t=0} =w_0, \quad \partial_t w\mid_{t=0} = w_1,
\end{equation*}
Now, exactly as in the proof of Proposition~\ref{th.1}, we obtain that
$$
\frac{d} {dt} {\mathcal E}(w(t)) \leq 
C \big({\mathcal E}(w(t))\big)^{1/2} 
 \Big(g(t) + f(t) \big({\mathcal E}(w(t))\big)^{1/2}  \Big),
 $$
 where
$$
g(t)= \|S(t)(\tilde{v}_0, \tilde{v}_1)\|^3_{L^6( \T^3)},\quad f(t) =  \|S(t)(\tilde{v}_0, \tilde{v}_1)\|_{L^\infty( \T^3)}.
$$
Therefore thanks to the Gronwall lemma, using that ${\mathcal E}(w(0))$ is well defined, we obtain the global existence for $w$.
Thus the solution of \eqref{shart} can be written as
 $$
 v(t)= S(t) (\tilde{v}_0, \tilde{v}_1) +w(t),\quad (w,\partial_t w)\in C(\R;{\mathcal H}^1).
 $$
Coming back to the definition of $\Theta$, we observe that 
$$
S(t)(\Theta)=\Theta.
$$
Thus $(v(t),\partial_t v(t))\in \Sigma$.
This completes the proof of Proposition~\ref{th.1.bis}.
\end{proof}
\section{Bounds on the possible growth of the Sobolev norms, $s>0$}
In this section, we are going to  follow the high-low decomposition method of Bourgain~\cite{BoIMRN}, or more precisely the reverse of Bourgain's method, as developed  for instance in the work by Gallagher and Planchon~\cite{GaPl}  (see also \cite{KPV}) or Bona and the second author \cite{Bona} 
and refine the global well-posedness results obtained in the previous sections, to prove~\eqref{eq.largetime1} when $s>0$. Notice that a similar strategy has been recently used by Colliander and Oh~\cite{CO} in the context of the well posedness of the cubic one-dimensional non linear Schr\"odinger equation below $L^2$. In the context of randomly forced parabolic-type equations, such sub-linear estimates appear rather naturally (see the woork by Kuksin and Shirikyan~\cite{KS})  
\par
Let us introduce the Dirichlet projectors to low/high frequencies.
If a function $u$ on the torus is given by its Fourier series 
$$ 
u(x)=a+\sum_{n\in\Z^3_{\star}}\Big(b_{n}\,\cos(n\cdot x)+c_{n}\sin(n\cdot x)\Big),
$$ 
for $N\geq 0$, an integer, we set
\begin{equation}\label{eq.proj}
\Pi_{N}(u)\equiv
a+\sum_{|n|\leq N}\Big(b_{n}\,\cos(n\cdot x)+c_{n}\sin(n\cdot x)\Big),
\quad 
\Pi^{N}(u)\equiv (1-\Pi_{N})(u).
\end{equation}
We also set $\Pi_{0}(u)\equiv a$.
The goal of this section is to prove the following statement.
\begin{proposition}\label{th.3}
Let $1>s>0$ and $\mu\in {\mathcal M}^s$. Consider the flow of \eqref{shart} established in the previous section.
Then for any $\varepsilon>0$ there exist $C, \delta >0$ such that for every $(v_0, v_1)\in \Sigma$, 
there exists $M>0$ such that the global solution to~\eqref{shart} constructed in the previous section satisfies 
$$ v(t)= S(t) \Pi^0(v_0, v_1)+ w(t), 
\qquad \|(w(t), \partial_t w(t)) \|_{\mathcal{H}^1(\T^3)} \leq C (M+ |t|)^{\frac {1-s} s + \varepsilon} ,
$$ 
with 
$ \mu (M>\lambda) \leq C e^{-\lambda^\delta}\,.
$
\end{proposition}
\begin{proof}
We only give the proof for positive times, the analysis for negative times being analogous.
For $\varepsilon>0$, $\delta >1/2$ and $\widetilde{\delta}>1/3$, we introduce the sets
 \begin{equation*}
 \begin{aligned}
 F_N&= \Big( (v_0, v_1)\in \Sigma\,:\, \|\Pi_N(v_{0}, v_{1})\|_{\mathcal{H}^1}\leq N^{1-s + \varepsilon}\Big),
 \\
  G_N&= \Big( (v_0, v_1)\in \Sigma\,:\, \|\Pi_N(v_{0})\|_{ L^4(\T^3)}\leq N^{\varepsilon}\Big),
  \\
  H_N&= \Big( (v_0, v_1)\in \Sigma\,:\, \|\langle t\rangle^{- \delta} S(t)(\Pi^N(v_{0}, v_{1}))\|_{L^2(\R_t ; L^\infty(\T^3))} \leq {N^{\varepsilon-s}}\Big),
  \\
 K_N&= \Big( (v_0, v_1)\in \Sigma \,:\,
  \|\langle t\rangle^{- \widetilde{\delta}} S(t)(\Pi^N(v_{0}, v_{1}))\|_{L^3(\R_t ; L^6(\T^3))} \leq  N^{\varepsilon-s}\Big).
 \end{aligned}
 \end{equation*}
 \begin{lemma}\label{eq.estim}
 Let $\delta>1$ and $\widetilde{\delta}>1/3$,
 There exists $\varepsilon_0>0$ such that for any $0<\varepsilon\leq \varepsilon_0$, there exists $C,c>0$ such that for every $N\geq 1$,
 \begin{equation*}
 \mu( F_N^c) \leq C e^{-cN^{2\varepsilon}}, \quad \mu( G_N^c) \leq C e^{-c N^{2\varepsilon}},\quad 
 \mu( H_N^c) \leq C e^{-c N^{\varepsilon}}, \quad \mu (K_N^c)\leq Ce^{-c N^{2\varepsilon}}\,.
 \end{equation*}
 \end{lemma}
 \begin{proof}
 For 
 $(u_0, u_1) \in \mathcal{H}^s( \T^3)$, we have 
 $$\|\Pi_N(u_0, u_1)\|_{\mathcal{H}^1( \T^3)} \leq CN^{1-s} \|(u_0, u_1)\| _{\mathcal{H}^s}\,. $$
Therefore according to~\eqref{eq.random2} 
$$
\mu( F_N^c) \leq C e^{-cN^{2\varepsilon}}\,.
$$
Next, using \eqref{eq.random1}, we infer that 
 $$
 \mu( G_N^c) \leq C e^{-c N^{2\varepsilon}}.
 $$
On the other hand we have 
 $$ \|\Pi^N(u_0, u_1)\|_{\mathcal{H}^0( \T^3)}\leq CN^{-s} \|(u_0, u_1)\|_{\mathcal{H}^s(\T^3)}. $$
Therefore using Remark~\ref{rem3} and Corollary~\ref{corA.6}, we obtain that
$$
 \mu( K_N^c) \leq C e^{-c N^{2\varepsilon}}, \qquad \mu( H_N^c) \leq C e^{-cN^{2\varepsilon}}.
$$
 This completes the proof of Lemma~\ref{eq.estim}.
\end{proof}
 Let us now define for $N\geq 1$ an integer, 
 $$
 E_N\equiv F_{N} \cap G_{N}\cap H_N\cap K_{N}\,.
 $$
 According to Lemma~\ref{eq.estim}, we have 
 $$ 
 \mu( E_N^c) \leq C e^{-c N^\kappa}, \qquad\kappa >0.
 $$
 Fix $\varepsilon_1>0$. Then we fix $\varepsilon>0$ small enough such that
\begin{equation}\label{eps1eps0}
\frac{1-s+ \varepsilon} {s- 2\varepsilon}\leq
\frac{1-s} {s}+\varepsilon_1, \qquad \varepsilon < \frac s 2. 
\end{equation}
Let us finally fix $\delta>1/2, \widetilde \delta > 1/3 $ such that
\begin{equation}\label{obrat}
 (\delta-\frac 1 2)s<2\delta\varepsilon, \qquad \widetilde \delta <1
 \end{equation}
We have the following statement.
 \begin{lemma}\label{iinntt}
For every $c>0$ there exists $C>0$ such that for every
$t\geq 1$, every integer $N\geq 1$ such that $t\leq cN^{s-2\varepsilon}$, every
$(v_0, v_1)\in E_N$ the solution of \eqref{shart} with data $(v_0,v_1)$ satisfies
$$
\|v(t)-S(t)\Pi^0(v_{0}, v_{1})\|_{\mathcal{H}^1( \T^3)}\leq CN^{1-s+\varepsilon}.
$$
In particular, thanks to \eqref{eps1eps0}, if $t\approx N^{s-2\varepsilon}$ then
$$
\|v(t)-S(t)\Pi^0(v_{0}, v_{1})\|_{\mathcal{H}^1( \T^3)}\lesssim t^{\frac{1-s} {s}+\varepsilon_1}.
$$
\end{lemma}
\begin{proof}
For $(v_0, v_1)\in E_N$ and we decompose the solution of \eqref{shart} with data $(v_0,v_1)$  as
$$
v(t)=S(t) \Pi^N(v_{0}, v_{1})+ w_{N},
$$
where $w_N$ solves the problem
\begin{equation}\label{problem_v}
(\partial_t^2-\Delta_{\T^3})w_N+(w_N+S(t) \Pi^N(v_{0}, v_{1}))^3=0,
\quad (w(0),\partial_t w(0))=\Pi_{N}(v_0,v_1).
\end{equation}
Using the energy estimates applied in the previous sections, we get the bound
\begin{equation}\label{eq.borne-N}
\frac{d} {dt} {\mathcal E}(w_N(t)) \leq 
C \big({\mathcal E}(w_N(t))\big)^{1/2} 
 \Big(g_N(t) + f_N(t) \big({\mathcal E}(w_N(t))\big)^{1/2}  \Big),
\end{equation}
where
$$
g_N(t)= \|S(t)\Pi^N(v_0, v_1)\|^3_{L^6( \T^3)},\quad f_N(t) =  \|S(t)\Pi^N(v_0, v_1)\|_{L^\infty( \T^3)}.
$$
Integrating \eqref{eq.borne-N}, we get
\begin{equation}\label{iinn}
{\mathcal E}^{1/2}(w_{N}(t)) \leq C e^{\int_0^t f_N(\tau)d\tau} 
\Big({\mathcal E}^{1/2}(w_{N}(0))+ \int_0^t g_N(\tau) d\tau\Big). 
\end{equation}
We now observe that for $(v_0,v_1)\in E_N$
\begin{equation*}
 \Big|\int_0^t g_N(\tau)d\tau \Big|\leq C N^{3(-s+ \varepsilon)}\langle t\rangle ^{3 \widetilde{\delta}}
 \leq 
 CN^{3(-s+ \varepsilon)+3\widetilde{\delta}(s-2\varepsilon)}\leq C,
 \end{equation*}
 provided
 $$
 -s+\varepsilon+\widetilde{\delta}(s-2\varepsilon)\leq 0.
 $$
 The last condition can be readily satisfied according to \eqref{obrat}.
 \par
 Next, we have (using Cauchy-Schwartz inequality in time) that for $(v_0,v_1) \in E_N$,
 \begin{equation*}
 \Big |\int_0^t f_N(\tau)d\tau \Big|\leq \|\langle \tau \rangle ^{- \delta} f_N\|_{L^2( \mathbb{R})} \langle t \rangle ^{\delta+ \frac 1 2}\leq CN^{-s + \varepsilon} \langle t\rangle^{\delta +\frac 1 2}
 \leq C N^{-s + \varepsilon+(\delta+ \frac 1 2 )(s-2\varepsilon)}\leq C,
 \end{equation*}
 provided $-s + \varepsilon+(\delta+ \frac 1 2) (s-2\varepsilon)\leq 0$, a condition which is satisfied thanks to \eqref{obrat}.
 \par
 
For $(v_0, v_1) \in E_N$, we have  
 $$ {\mathcal E}^{1/2} (w_{N}(0))\leq C( \|\Pi_N( v_{0}, v_{1})\|_{{\mathcal H}^1}+\|\Pi_{N}(v_0)\|_{L^4}^2)  \leq CN^{1-s+ \varepsilon}
 $$
 and coming back to \eqref{iinn}, we get 
 \begin{equation*}
{\mathcal E}^{1/2}(w_{N}(t)) \leq C N^{1-s+ \varepsilon}.
\end{equation*}
Recall that
$$ 
v(t)= w_{N}(t) + S(t)\Pi^N(v_{0}, v_{1})=S(t)\Pi^0(v_{0}, v_{1})+w_{N}(t)-S(t)\Pi_N\Pi^0(v_{0}, v_{1}).
$$
We have that for a solution to the linear wave equation  the linear energy 
$$ \|\nabla_x u \|_{L^2( \T^3)}^2 + \|\partial_t u \|_{L^2( \T^3)}^2
$$ is independent of time and that if $(u, \partial_tu) $ is orthogonal to constants ($(u, \partial_t u) =  \Pi^0 (u, \partial_t u)$), then this energy controls the $\mathcal{H}^1( \T^3)$-norm, we deduce for $(v_0,v_1)\in E_{N}\subset F_N$ that
$$
\|S(t)\Pi_{N}\Pi^0(v_{0}, v_{1})\|_{\mathcal{H}^1( \T^3)}\leq CN^{1-s+\varepsilon}
$$
and therefore
$$
\|v(t)-S(t)\Pi^0(v_{0}, v_{1})\|_{\mathcal{H}^1( \T^3)}\leq CN^{1-s+\varepsilon}\,.
$$
This completes the proof of Lemma~\ref{iinntt}.
\end{proof}
Next we set
$$
E^{N}=\bigcap_{M\geq N} E_{M},
$$ 
where the intersection is taken over the dyadic values of $M$, i.e. $M=2^j$ with $j$ an integer.
Thus $\mu(E^N)$ tends to $1$ as $N$ tends to infinity.
Using Lemma~\ref{iinntt}, we obtain that  there exists $C>0$ such that for every $t\geq 1$, every $N$, every $(v_0,v_1)\in E^N$,
$$
\|v(t)-S(t)\Pi^0(v_{0}, v_{1})\|_{\mathcal{H}^1 ( \T^3)}\leq C\big(N^{1-s+\varepsilon}+t^{\frac{1-s} {s}+\varepsilon_1}\big)\,.
$$
Finally, we set 
 $$
 E= \bigcup_{N= 1}^{\infty} E^N\,.
 $$
 We have thus shown the $\mu$ almost sure bounds on the possible growths of the Sobolev norms of the solutions established in the previous section
 for data in $E$ which is of full $\mu$ measure.  This completes the proof of Proposition~\ref{th.3}.
\end{proof}

\section{Global existence, the case $s=0$}
In the previous sections, to prove global existence for $s>0$ we used at some points the Sobolev embeddings $W^{\varepsilon, p} \subset L^\infty$ for $\varepsilon >0$ and $p$ large enough. This argument implied an $\varepsilon$ loss in the estimates. Here, such a loss is forbidden and we are going to follow instead a probabilistic version of the strategy iniciated by Yudovich~\cite{Yu} to prove global existence for two dimensional Euler  equation 
(see also Brezis-Gallouet~\cite{BrGa} and Burq-G\'erard-Tzvetkov~\cite{BGT1} for similar ideas in the context of dispersive equations). 
Our goal in this section is to prove the following statement.
\begin{proposition}\label{th.4}
Let $\mu\in {\mathcal M}^0$.
Then there exists $C>0$ such that for $\mu$ almost every $(v_0, v_1) \in L^2( \T^3) \times H^{-1} ( \T^3)$, there exists a unique global solution 
$$v(t) \in S(t)\Pi^0(v_0, v_1) + C(\R; H^1(\T^3) \times L^2(\T^3))$$ of the non linear wave equation~~\eqref{shart}.
Moreover there exists $M=M(v_0,v_1)>0$ such that the solution furthermore satisfies 
$$  v(t)= S(t) \Pi^0(v_0, v_1)+ w(t), \qquad 
\|(w(t), \partial_t w(t)) \|_{\mathcal{H}^1(\T^3)}\leq C e^{C(t+M)^2} ,
$$ 
with 
$ \mu (M>\lambda) \leq C e^{-\lambda^\delta}$ for some $\delta>0$.
\end{proposition}
\begin{proof}
 Again, we shall only consider positive times.
 Let us  notice that according to  Proposition~\ref{random1} (with $p_1= p_2= j, \delta_j = \frac 2 j$), 
 there exists $C>0$ such that for any $j\geq 2$, we have 
\begin{equation*}
  \mu\big( (v_0, v_1)\in {\mathcal H}^0\,:\, 
  \|  
  \langle t \rangle ^{-\delta_j}\Pi^0
  S(t)  ( v_0, v_1) \|_{L^j(\R_t \times \T^3)} >\lambda \big)
 \leq \Bigl(C\frac{ \sqrt{j} (\delta_j j -1)^{-\frac {1} {j}}}
 {\lambda }\Bigr) ^{j}  = \Bigl(\frac{ C \sqrt{j}} \lambda\Big) ^{j}.
\end{equation*}
On the other hand, according to Corollary~\ref{random3} (with $p_1=3, p_2=6, \delta =1>\frac 1 3 $), 
\begin{equation*}
 \mu \big( 
 (v_0, v_1)\in {\mathcal H}^0\,:\, 
  \|\langle t \rangle^{-1} \Pi^0 S(t)  ( v_0, v_1) \|_{L^3(\R_t ; L^6(\T^3))} >\lambda \big) \leq C e^{-c  \lambda^2}\,.
 \end{equation*}
For any integer $k\geq 1$, we set (notice that for the global existence part in the case $s=0$ in Theorem~\ref{main}, we only use the case $k=1$, the other cases will only be used to prove the invariance of our set by the flow),
\begin{equation*}
F_{j,k}= 
\Big((v_0, v_1)\in {\mathcal H}^0\,:\,  \|  \langle t\rangle^{-\delta_j}\Pi^0S(t) ( v_0, v_1) \|_{L^j(\R_{t} \times \T^3)} \leq 2kC\sqrt{j}\Big)
\end{equation*}
and
\begin{multline*}
G_{j,k}= 
\Big((v_0, v_1)\in {\mathcal H}^0\,:\, 
 \\
 \|
 \langle t\rangle^{-1}
 \Pi^0  S(t) ( v_0, v_1) \|^3_{L^3(\R_{t} ; L^6(T^3))} 
 +
 \|\Pi_{0}(v_0,v_1)\|_{{\mathcal H}^1}+\|\Pi_{0}v_0\|_{L^4(\T^3)}^2
 \leq kj
\Big).
\end{multline*}
Therefore we have 
\begin{equation}\label{ejk}
\mu(F_{j,k}^c)+\mu(G_{j,k}^c)\lesssim e^{-cj^\delta}\,,\quad \delta>0.
\end{equation}
Next we set
$
E_{j,k}= G_{j,k} \cap F_{j,k}. 
$
For any $(v_0, v_1) \in E_{j,k}$, we write
 $$
 v(t)= S(t)\Pi^0 (v_0, v_1)+w(t)\,,\quad (w(0),\partial_t w(0))=\Pi_{0}(v_0,v_1).
 $$ 
Using the energy estimate, already performed several times in the previous sections, we get
\begin{multline*}
\frac{d} {dt} {\mathcal E}(w(t)) \leq C ({\mathcal E}(w(t)))^{1/2}  \|w^3-  (\Pi^0S(t)(v_0, v_1)+ w)^3\|_{L^2( \T^3)}
\\
\leq C ({\mathcal E}(w(t)))^{1/2} \Bigl(\|\Pi^0S(t)(v_0, v_1)\|^3_{L^6( \T^3)} + 
\|\Pi^0S(t)(v_0, v_1)\|_{L^j( \T^3)} \|w^2\|_{L^{\widetilde{j}}( \T^3)} \Bigr), 
\end{multline*}
provided
$$
\frac 1 j + \frac 1 {\widetilde{j}}= \frac 1 2\,.
$$
On the other hand
$$ \|v\|_{L^{2\widetilde{j}}}\leq \|v\|_{L^4}^{1-\theta} \|v\|_{L^6}^{\theta}, \qquad \frac{1-\theta} 4 + \frac {\theta} {6} = \frac 1 {2\widetilde {j}}\Rightarrow \theta = \frac 6 j $$
and thus we obtain
\begin{multline*}
\frac{d} {dt} {\mathcal E}(w(t)) 
\leq 
C ({\mathcal E}(w(t)))^{1/2}
\\
\times
 \Bigl(\|\Pi^0S(t)(v_0, v_1)\|^3_{L^6( \T^3)} + 
\|\Pi^0S(t)(v_0, v_1)\|_{L^j( \T^3)} ({\mathcal E}(w(t))^{\frac 1 2 (1+\frac 6 j)}\Bigr).
\end{multline*}
Consequently, as long as $({\mathcal E}(w(t)))^{1/2}$ remains smaller than $2^j$, we have $\big({\mathcal E}(w(t))\big)^{\frac 3 j} \leq C$ 
and consequently by Gronwall inequality,
for $(v_0,v_1)\in E_{j,k}$ and $t\lesssim \sqrt{j}$,
\begin{multline*}
 ({\mathcal E}(w(t)))^{1/2} \leq
 \\  
 Ce^{\int_0^t \|\Pi^0S(\tau)(v_0, v_1)\|_{L^j( \T^3)} d\tau }
\Big(   \int_0^t \|\Pi^0S(\tau)(v_0, v_1)\|^3_{L^6( \T^3)} d\tau
+
({\mathcal E}(w(0)))^{1/2} 
\Big)
\\
\leq 
Cj^{\frac 3 2} e^{c\|(1-\Pi_{0})S(t)(v_0, v_1)\|_{L^j(0,t)\times \T^3)}\times \langle t \rangle^{1- \frac 1 j } }\,.
\end{multline*}
Next, we can write for $t\lesssim \sqrt{j}$ and $(v_0,v_1)\in E_{j,k}$
$$
\|(1-\Pi_{0})S(t)(v_0, v_1)\|_{L^j(0,t)\times \T^3)}
\leq
C
j^{\frac{1}{2\delta_j}}
\|\langle t\rangle^{-\delta_j}(1-\Pi_{0})S(t)(v_0, v_1)\|_{L^j(\R\times \T^3)}
$$
and consequently there exists a small $\alpha>0$ and a large $j_0$ depending only on $k$ such that for $j\geq j_0$,
$t\leq \alpha \sqrt{j}$ and $(v_0,v_1)\in E_{j,k}$, as long as $({\mathcal E}(w(t)))^{1/2}$ remains bounded by $2^j$ and $t\leq \alpha \sqrt{j}$, we have 
$$ 
({\mathcal E}(w(t)))^{1/2} \leq 
Cj^{\frac 3 2} e^{2C \sqrt{j}\times |t|}\leq Cj e^{2C \alpha j}  \leq 2^j\,.
$$
From the usual bootstrap argument, we deduce that the bootstrap assumption $({\mathcal E}(w(t)))^{1/2}  \leq 2^j$ remains satisfied for $t \leq \alpha \sqrt{j}$.
Thus, we obtain that
$$ \forall k \geq 1, \exists \alpha>0, \exists j_0>0; \forall\, j\geq j_0,\,\,\forall\, (v_0, v_1) \in E_{j,k},\quad
({\mathcal E}(w(t)))^{1/2} \leq  2^j, \qquad t \leq \alpha \sqrt{j}.
$$
Next, we set
$$
E^j_k= \bigcap_{N\geq \max(j,j_0)} E_{N,k}, 
$$
where the intersection is taken over the values of $N\geq \max(j,j_0)$.
Thus, according to~\eqref{ejk}, 
$$
\mu( (E^j_k)^c) \leq \sum_{n\geq \max(j,j_0)} Ce^{-c n^\delta}\leq C e^{-j^\delta} 
$$ 
and consequently, the set $\Sigma_k$, defined as
$$
\Sigma_k=\bigcup_{j=1}^{\infty}E^j_{k}
$$
is of full $\mu$ measure.

Let $(v_0, v_1) \in E^j_k$.
The solution is already defined up to time $\alpha\max (\sqrt{j_0},\sqrt{j})$.
For $t \geq \alpha \max(\sqrt{j_0},\sqrt{j})$, there exists a dyadic $N\geq \max(j_0,j)$ such that 
$\alpha \sqrt{N} \leq t \leq \alpha \sqrt{2N}$ and we deduce that the solution is defined up to time $\alpha \sqrt{N2N}$ with bounds 
$$
({\mathcal E}(w(t)))^{1/2}  \leq 2^{2N} \leq Ce^{c(1+ t^2)}.
$$
Hence, the solution is globally defined.
The above discussion also gives the claimed bound on the possible growth of the Sobolev norms.
Namely, we have that for every $t\geq 1$, every $(v_0, v_1) \in E^j_k$,
$$
({\mathcal E}(w(t)))^{1/2}  \leq Ce^{c(j+j_0+ t^2)}\,.
$$
Thus we have the needed bound for every $(v_0,v_1)\in\Sigma_k$.
This ends the proof of Proposition~\ref{th.4}.
\end{proof}
Let us now define a set of full measure invariant under the dynamics established in Proposition~\ref{th.4}.
Set
$$
\Sigma=\bigcap_{k=1}^{\infty}\Sigma_{k}.
$$
Then $\Sigma$ is of full measure.
Coming back to the definition of $\Sigma_k$, we obtain that for every $t$ there exists $k_0$ such that for every $k$,
$
S(t)(\Sigma_k)\subset \Sigma_{k+k_0}
$
which in turn implies that 
$
S(t)(\Sigma)=\Sigma.
$
It remains to observe that the argument of the proof of Proposition~\ref{th.4} implies that the set
$
\Sigma+{\mathcal H}^1
$
is invariant under the dynamics (and of full measure).
This completes the proof of Theorem~\ref{main}.
\section{Probabilistic continuity of the flow.}
The purpose of this section is to prove Theorem~\ref{th_continuity}.
According to Proposition~\ref{ochak}, we have that for any $2\leq p_1<+\infty, 2\leq p_2 \leq + \infty, \delta > 1+ \frac 1 {p_1}$,  $\eta\in(0,1)$, $\alpha\in (0,1)$ and $\beta>0$,
\begin{multline*}
\mu\otimes\mu\Big(
(V_0,V_1)\in {\mathcal H}^s\times {\mathcal H}^s\,:\,\
\|\langle t \rangle ^{- \delta} S(t) (V_0- V_1)\|_{L^{p_1} (  \R_t ; L^{p_2}( \T^3))}> 
\eta^{1-\alpha}
\\
{\,\rm\,or\,}
\|\langle t \rangle ^{- \delta} S(t) (V_j )\|_{L^{p_1} (  \R_t ; L^{p_2}( \T^3))}> 
\beta \log\log(\eta^{-1}),\, j=0,1
\Big\vert
\\
 \|V_0- V_1\|_{\mathcal{H}^s(\T^3)}< \eta
 {\rm\,\,\,and\,\,\,} \|V_j\|_{\mathcal{H}^s(\T^3)}\leq A,\,j=0,1  \Big) 
\longrightarrow 0,
\end{multline*}
as $\eta\rightarrow 0$. Therefore, we can also suppose that 
\begin{equation}\label{yu1}
\|\langle t \rangle ^{- \delta} S(t) (V_0- V_1)\|_{L^{p_1} (  \R_t ; L^{p_2}( \T^3))}\leq
\eta^{1-\alpha}
\end{equation}
and
\begin{equation}\label{yu2}
\|\langle t \rangle ^{- \delta} S(t) (V_j )\|_{L^{p_1} (  \R_t ; L^{p_2}( \T^3))}\leq
\beta \log\log(\eta^{-1}),\, j=0,1,
\end{equation}
when estimate the needed conditional probability.

We therefore need to estimate the difference of two solutions under the assumptions \eqref{yu1} and \eqref{yu2}, in the regime $\eta\ll 1$.
Let 
$$
v_{j}(t)=S(t)(V_j)+w_{j}(t), \quad j=0,1
$$
be two solutions of the cubic wave equation with data $V_j$ (and thus $(w_{j}(0),\partial_{t}w_{j}(0))=(0,0)$).
Applying the energy estimate, performed several times in this paper, we get the bound
$$
\frac{d}{dt}E^{1/2}(w_{j}(t))\leq
C\Big(
\|S(t)(V_j)\|^3_{L^{6}(\T^3)}
+
\|S(t)(V_j)\|_{L^{\infty}(\T^3)}E^{1/2}(w_{j}(t))
\Big),\quad j=0,1,
$$
and therefore, under the assumptions \eqref{yu1} and \eqref{yu2},
\begin{equation}\label{kam}
E^{1/2}(w_{j}(t))\leq 
C_{T}\,e^{C_T\beta \log\log(\eta^{-1})}\log\log(\eta^{-1})\leq
C_{T}[\log(\eta^{-1})]^{C_{T}\beta},\quad t\in [0,T],
\end{equation}
where here and in the sequel we denote by $C_T$ different constants depending only on $T$ (but independent of $\eta$).
\par
We next estimate the difference $w_0-w_1$. Using the equations solved by $w_0$, $w_1$, we infer that
\begin{multline*}
\frac{d}{dt}\|w_0(t,\cdot)-w_1(t,\cdot)\|^2_{{\mathcal H}^1(\T^3)}
 \leq  2
\Big|
\int_{\T^3}\partial_{t}(w_0(t,x)-w_1(t,x))(\partial_t^2-\Delta)(w_0(t,x)-w_1(t,x))dx
\Big|
\\
 \leq 
C
\|w_0(t,\cdot)-w_1(t,\cdot)\|_{{\mathcal H}^1(\T^3)}
\|(w_0+S(t)(V_0))^3-w_0^3-(w_1+S(t)(V_1))^3+w_1^3\|_{L^2(\T^3)}\,.
\end{multline*}
Therefore using the Sobolev embedding $H^1(\T^3)\subset L^6(\T^3)$, we get
\begin{multline*}
\frac{d}{dt}\|w_0(t,\cdot)-w_1(t,\cdot)\|_{{\mathcal H}^1(\T^3)}
 \leq 
C\Big(\|w_0(t,\cdot)-w_1(t,\cdot)\|_{{\mathcal H}^1(\T^3)}+\|S(t)(V_0-V_1)\|_{L^6(\T^3)}\Big)
\\
\Big(\|w_0(t,\cdot)\|_{L^6(\T^3)}^2+\|w_1(t,\cdot)\|_{L^6(\T^3)}^2+\|S(t)(V_0)\|_{L^6(\T^3)}^2+\|S(t)(V_1)\|_{L^6(\T^3)}^2\Big).
\end{multline*}
Therefore, using \eqref{kam} and the Gronwall lemma, under the assumptions \eqref{yu1} and \eqref{yu2}, for $t\in [0,T]$,
$$
\|w_0(t,\cdot)-w_1(t,\cdot)\|_{{\mathcal H}^1(\T^3)}
\leq C_{T}\eta^{1-\alpha}[\log(\eta^{-1})]^{C_{T}\beta}\,
e^{C_{T}[\log(\eta^{-1})]^{C_{T}\beta}}\leq
C_{T}\eta^{1-\alpha-C_{T}\beta}\leq C_{T}\eta^{1/2},
$$
provided $\alpha,\beta \ll 1$. In particular by the Sobolev embedding
$$
\|w_0-w_1\|_{L^4([0,T]\times\T^3)}\leq C_{T}\eta^{1/2},
$$
and therefore under the assumption \eqref{yu1},
$$
\|v_0-v_1\|_{L^4([0,T]\times\T^3)}\leq C_{T}\eta^{1/2}\,.
$$
In summary, we obtained that for a fixed $\varepsilon>0$, the $\mu\otimes\mu$ measure of $V_0$, $V_1$ such that
$$
\|\Phi(t)(V_0)-\Phi(t)(V_1)\|_{X_{T}}>\varepsilon
$$
under the conditions \eqref{yu1}, \eqref{yu2}  and $\|V_0-V_1\|_{{\mathcal H}^s}<\eta$  is zero, as far as $\eta>0$ is sufficiently small.
Therefore, we obtain that the left hand side of \eqref{dimanche} tends to zero as $\eta\rightarrow 0$.
This ends the proof of the first part of Theorem~\ref{th_continuity}.
\par
For the second part of the proof of Theorem~\ref{th_continuity}, we argue by contradiction. 
Suppose thus that for every $\varepsilon>0$ there exist $\eta>0$ and $\Sigma$ of full $\mu$ measure such that 
$$
\forall\, V,V'\in\Sigma\cap B_{A},\, \|V-V'\|_{{\mathcal H}^s}<\eta
\implies\,\,
 \| \Phi(t) ( V) - \Phi(t) ( V') \|_{X_T} <\varepsilon.  
$$
Let us apply the previous affirmation with $\varepsilon=1/n$, $n=1,2,3\dots$ which produces full measure sets $\Sigma(n)$. 
Set
$$
\Sigma_1\equiv\bigcap_{n=1}^{\infty}\Sigma(n).
$$
Then $\Sigma_1$ is of full $\mu$ measure and we have that
$$
\forall\,\varepsilon>0,\,
\exists\,\eta>0,\,\,\,
\forall\, V,V'\in\Sigma_1\cap B_{A},\, \|V-V'\|_{{\mathcal H}^s}<\eta
\implies\,\,
 \| \Phi(t) (V) - \Phi(t) ( V') \|_{X_T} <\varepsilon.  
$$
In other words the (nonlinear) map $\Phi(t)$ from ${\mathcal H}^s$ to $X_T$, restricted to $\Sigma_1\cap B_{A}$, is uniformly continuous.  
Therefore it can be extended in a unique way to a uniformly continuous map on $\overline{\Sigma_1\cap B_{A}}$.
Since we supposed that the support of $\mu$ is the whole ${\mathcal H}^s$, we obtain that $\overline{\Sigma_1\cap B_{A}}=B_A$.
Let us denote by $\overline{\Phi(t)}$ the extension of $\Phi(t)$ to $B_A$.  We therefore have
\begin{equation}\label{rex}
\forall\,\varepsilon>0,\,
\exists\,\eta>0,\,\,\,
\forall\, V,V'\in
B_{A},\, \|V-V'\|_{{\mathcal H}^s}<\eta
\implies\,\,
 \| \overline{\Phi(t)} ( V) - \overline{\Phi(t)} ( V') \|_{X_T} <\varepsilon.  
\end{equation}
We have the following lemma.
\begin{lemma}\label{rex1}
For $V\in (C^{\infty}(\T^3)\times C^{\infty}(\T^3))\cap B_A$, we have that $\overline{\Phi(t)}(V)=(u,u_t)$, where $u$ is the unique classical solution on $[0,T]$ of
$$
(\partial_{t}^2-\Delta)u+u^3=0,\quad (u(0),\partial_{t}u(0))=V. 
$$
\end{lemma}
\begin{proof}
Let us first show that that first component of
$\overline{\Phi(t)}(V)\equiv(\overline{\Phi_1(t)}(V),\overline{\Phi_2(t)}(V))$ is a solution of the cubic wave equation.
Observe that by construction, necessarily $\overline{\Phi_2(t)}(V)=\partial_t \overline{\Phi_1(t)}(V)$ in the distributional sense (in ${\mathcal D}'((0,T)\times\T^3)$).

We have that 
$$
V=\lim_{n\rightarrow\infty}V_{n}\,,
$$
in ${\mathcal H}^s$ with  $V_n\in \Sigma_1\cap B_{A}$.
We also have that
\begin{equation}\label{limit}
(\partial_{t}^2-\Delta)(\Phi_1(t)(V_n))+(\Phi_1(t)(V_n))^3=0, 
\end{equation}
with the notation $\Phi(t)=(\Phi_1(t),\Phi_2(t))$. In addition, 
$$
\overline{\Phi(t)}(V)=\lim_{n\rightarrow\infty}\Phi(t)(V_{n})\,,
$$
in $X_T$.
We therefore have that
$$
(\partial_{t}^2-\Delta)(\overline{\Phi_1(t)}(V))=\lim_{n\rightarrow\infty}(\partial_{t}^2-\Delta)(\Phi_1(t)(V_n)),
$$
in the distributional sense. Moreover, coming back to the definition of $X_T$, we also obtain that
$$
(\overline{\Phi_1(t)}(V))^3=\lim_{n\rightarrow\infty}(\Phi_1(t)(V_n))^3,
$$
in $L^{4/3}([0,T]\times\T^3)$.
Therefore, passing into the limit $n\rightarrow\infty$ in (\eqref{limit}), we obtain that $\overline{\Phi_1(t)}(V)$ solves the cubic wave equation (with data $V$).
Moreover, since $(\overline{\Phi_1(t)}(V))^3\in L^{4/3}([0,T]\times\T^3)$, it also satisfies the Duhamel formulation of the equation.
\par
Let us denote by $u(t)$, $t\in [0,T]$ the classical solution of 
$$
(\partial_{t}^2-\Delta)u+u^3=0,\quad (u(0),\partial_{t}u(0))=V, 
$$
defined by Theorem~\ref{th1}.
Set $v\equiv \overline{\Phi_1(t)}(V)$. Since our previous analysis has shown that $v$ is a solution of the cubic wave equation, we have that 
\begin{equation}\label{razlika}
(\partial_{t}^2-\Delta)(u-v)+u^3-v^3=0,\quad (u(0),\partial_{t}u(0))=(0,0)\,.
\end{equation}
We now invoke the $L^4-L^{4/3}$ non homogenous estimates for the three dimensional wave equation.
Namely, we have that there exists a constant (depending on $T$)
such that for every interval $I\subset [0,T]$, the solutions of the wave equation
$$
(\partial_{t}^2-\Delta)w=F,\quad (u(0),\partial_{t}u(0))=(0,0)
$$
satisfies
\begin{equation}\label{strichartz}
\|u\|_{L^4(I\times\T^3)}\leq C\|F\|_{L^{4/3}(I\times\T^3)}\,.
\end{equation}
Applying \eqref{strichartz} in the context of \eqref{razlika} together with the H\"older inequality yields the bound
\begin{equation}\label{strichartz2}
\|u-v\|_{L^4(I\times\T^3)}\leq C
\big(\|u\|_{L^{4}(I\times\T^3)}^2+\|v\|_{L^4(I\times\T^3)}^2\big)
\|u-v\|_{L^{4}(I\times\T^3)}\,.
\end{equation}
Since $u,v\in L^{4}(I\times\T^3)$, we can find a partition of intervals $I_1,\dots,I_{l}$ of $[0,T]$ such that
$$
C\big(\|u\|_{L^{4}(I_j\times\T^3)}^2+\|v\|_{L^4(I_j\times\T^3)}^2\big)<\frac{1}{2},\quad j=1,\dots, l.
$$
We now apply \eqref{strichartz2} with $I=I_j$, $j=1,\dots, l$ to conclude that $u=v$ on $I_1$, then on $I_2$ and so on up to $I_l$ which gives that $u=v$ on 
$[0,T]$. 
Thus $u=\overline{\Phi_1(t)}(V)$ and therefore also $\partial_t u=\overline{\Phi_2(t)}(V)$.
This completes the proof of Lemma~\ref{rex1}.
\end{proof}
It remains now to apply Lemma~\ref{rex1} to the sequence of smooth data in the statement of Theorem~\ref{th1} to get a contradiction with \eqref{rex}.
More precisely, if $(U_n)$ is the sequence involved in the statement of Theorem~\ref{th1}, the result of Theorem~\ref{th1} affirms that 
$\overline{\Phi(t)}(U_n)$ tends to infinity in $L^{\infty}([0,T];{\mathcal H}^s)$ while \eqref{rex} affirms that the same sequence tends to zero in the same space 
$L^{\infty}([0,T];{\mathcal H}^s)$. 
\appendix 
\section{Random series}
In this appendix, we collected the various results we need about random series. 
Most of them are well known in slightly different contexts, and the proofs we give are essentially adaptations of the classical proofs. 
The conditioned versions of our estimates (see Section~\ref{sec.random2}), though very natural do not seem to appear in the literature.
\subsection{Basic large deviation estimates}
\begin{proposition}\label{random1}
Let us fix $\mu\in {\mathcal M}^s$, $s\in [0,1)$ and let us
suppose that $\mu$ is induced via the map \eqref{eq.proba} from the couple $(u_0,u_1)\in {\mathcal H}^s$.
Then there exists a positive constant $C$ such that for every $2\leq p_1, p_2\leq q<+\infty$ and every $\delta >\frac  1 {p_1}$,  
\begin{multline}\label{tatu}
\mu\Big((v_0,v_1)\in {\mathcal H}^s\,:\,
\|\langle t \rangle ^{- \delta} (1-\Pi_{0})S(t) (v_0,v_1)\|_{L^{p_1} (  \R_t ; L^{p_2}( \T^3))}> \lambda 
\Big)
\\
\leq \Bigl(C\frac{ \sqrt{q}\|(u_0, u_1)\|_{\mathcal{H}^0(\T^3)}(\delta p_1 -1)^{-\frac 1 {p_1}}}{\lambda}\Bigr) ^{q}
\end{multline}
\end{proposition}
\begin{proof}
By definition, the left hand-side of \eqref{tatu} equals
$$
p\Big(\omega\in\Omega\,:\,\|\langle t \rangle ^{- \delta} (1-\Pi_{0})S(t) (u_0^\omega,u_1^\omega)\|_{L^{p_1} (  \R_t ; L^{p_2}( \T^3))}> \lambda \Big).
$$
We decompose
\begin{multline*}
\Pi^0S(t) (u_0^\omega,u_1^{\omega})=\sum_{n\in\Z^3_{\star}}
\Big(\big(\beta_{n,0}(\omega)b_{n,0}\cos(t|n|)+\beta_{n,1}(\omega)b_{n,1}\frac{\sin(t|n|)}{|n|}\big)\cos(n\cdot x)
\\
+\big(\gamma_{n,0}(\omega)c_{n,0}\cos(t|n|)+\gamma_{n,1}(\omega)c_{n,1}\frac{\sin(t|n|)}{|n|}\big)\sin(n\cdot x)\Big),
\end{multline*}
with
$$
\sum_{n\in\Z^3_{\star}}
\Big(|b_{n,0}|^2+|c_{n,0}|^2+|n|^{-2}(|b_{n,1}|^2+|c_{n,1}|^2)\Big)
\leq C\|(u_0, u_1)\|_{\mathcal{H}^0(\T^3)}^2\,.
$$
Now, using the triangle inequality, by writing $\cos(n\cdot x)$ and $\sin(n\cdot x)$ as linear combination of $\exp(\pm i (n\cdot x))$, we observe that it suffices to get
the bound
$$ 
p ( \omega\,:\,\| \langle t \rangle^{- \delta }
\sum_{n} d_{n}(t) c_n  g_n^\omega e^{i n \cdot x}\|_{L^{p_1}( \R_t; L^{p_2}( \T^3))} > \lambda ) \leq 
\Bigl(\frac { C \sqrt {q}\|(c_n)\|_{l^2} (\delta p_1 -1)^{-\frac 1 {p_1}}} { \lambda} \Bigr)^{q},
$$
where $(g_n^\omega)$ are independent real random variables with joint distribution satisfying \eqref{subgauss} and $|d_n(t)|\leq 1$.
Using the Minkowski inequality, we can write for $q\geq p$, 
\begin{multline*}
\| \langle t \rangle^{- \delta }\sum_{n} 
d_{n}(t)c_n g_n^\omega e^{i n \cdot x}\|_{L^q( (\Omega; L^{p_1}( \R_t; L^{p_2}( \T^3)))}\\
\begin{aligned}&\leq &\| \langle t \rangle^{- \delta }\sum_{n} 
d_{n}(t)c_n g_n^\omega e^{i n \cdot x}\|_{L^{p_1}( \R_t; L^{p_2}( \T^3); L^q (\Omega))}\\
&=  &\| \|\langle t \rangle^{- \delta }\sum_{n} d_{n}(t)c_n g_n^\omega e^{i n \cdot x}\|_{L^q(\Omega)}\|_{L^{p_1}( \R_t; L^{p_2}( \T^3))}\,.
\end{aligned}
\end{multline*}
By using \cite[Lemma~3.1]{BT1}, we get
\begin{multline*}
\| \langle t \rangle^{- \delta }\sum_{n} d_n(t)c_n g_n^\omega e^{i n \cdot x}\|_{L^q( (\Omega; L^{p_1}( \R_t; L^{p_2}( \T^3)))}\\
\begin{aligned}&\leq \| C\sqrt{q} \Bigl(\sum_{n} \bigl|\langle t \rangle^{- \delta }d_n(t) c_n  e^{i n \cdot x}\bigr|^2 \Bigr)^{1/2}\|_{L^{p_1}( \R_t; L^{p_2}( \T^3))}\\
&\leq C\sqrt{q} \|\sum_{n} \langle t \rangle^{- 2\delta } |c_n|^2  \|^{1/2}_{L^{p_1/2}( \R_t; L^{p_2/2}( \T^3))}\\
&\leq C\sqrt{q} \Bigl(\sum_{n} \|\langle t \rangle^{- 2\delta }\|_{L^{p_1/2}( \R_t)}|\alpha _n|^2 \Bigr)^{1/2}\\
&\leq C\sqrt{q} (\delta p_1-1)^{-1/p_1}\Bigl(\sum_n|\alpha_n|^2\Bigr)^{1/2} 
\end{aligned}
\end{multline*}
and we conclude the proof of Proposition~\ref{random1} by using the Tchebichev inequality
$$ p (\omega\,:\, |A(\omega)|>\lambda) )\leq \lambda^{-q}\|A\|_{L^q(\Omega)}^q\,.$$
\end{proof}
For fixed $p_1, p_2$, we can optimize the estimate by taking 
\begin{equation}\label{eq.random3}
C\frac{ \sqrt{q}\|(u_0, u_1)\|_{\mathcal{H}^0(\T^3)}}{\lambda}
= \frac 1 2 \Leftrightarrow  q= \frac{ \lambda^2 \|(u_0, u_1)\|^{-2}_{\mathcal{H}^0(\T^3)} } {4C^2} ,
\end{equation}
and we deduce the following statement.
\begin{corollary}\label{random3}
There exist $C,c>0$ such that under the assumptions of  Proposition~\ref{random1} for every $\lambda>0$,
$$
\mu\Big((v_0,v_1)\in {\mathcal H}^s\,:\,
\|\langle t \rangle ^{- \delta}\Pi^0 S(t) (v_0,v_1)\|_{L^{p_1} (  \R_t ; L^{p_2}( \T^3))}> \lambda 
\Big)\leq C\exp\Bigl(- \frac {c\lambda^2} {\|(u_0, u_1)\|^2_{\mathcal{H}^0(\T^3)}}\Bigr)\,.
$$
\end{corollary}
\begin{remark}\label{rem3}
Notice that the measure $\mu$ is a tensor product of two probability measures $\mu_N$ and $\mu^N$ defined on the images of the projectors $\Pi_N$ and $\Pi^N$ respectively. As a consequence, applying Corollary~\ref{random3} to the measure $\mu^N$, we get that under the assumptions of  Proposition~\ref{random1} for every $\lambda>0$,
\begin{multline}
\mu\Big((v_0,v_1)\in {\mathcal H}^s\,:\,
\|\langle t \rangle ^{- \delta}\Pi^N S(t) (v_0,v_1)\|_{L^{p_1} (  \R_t ; L^{p_2}( \T^3))}> \lambda 
\Big)\\
= \mu^N\Big((v_0^N,v_1^N)\in \Pi^N({\mathcal H}^s)\,:\,
\|\langle t \rangle ^{- \delta} S(t) (v_0^N,v_1^N)\|_{L^{p_1} (  \R_t ; L^{p_2}( \T^3))}> \lambda 
\Big)\\
\leq C\exp\Bigl(- \frac {c\lambda^2} {\|\Pi^N(u_0, u_1)\|^2_{\mathcal{H}^0(\T^3)}}\Bigr)\leq  C\exp\Bigl(- \frac {c\lambda^2} {N^{-2s} \|\Pi^N(u_0, u_1)\|^2_{\mathcal{H}^s(\T^3)}}\Bigr)\,.
\end{multline}
\end{remark}
Notice now that if $u_0$ and $ u_1$ are constant, the free evolution is 
$$ S(t) (u_0, u_1)= u_0 + u_1t. $$
Therefore we deduce the following statement.
\begin{corollary} \label{random2}
Let us fix $\mu\in {\mathcal M}^s$, $s\in [0,1)$ and let us
suppose that $\mu$ is induced via the map \eqref{eq.proba} from the couple $(u_0,u_1)\in {\mathcal H}^s$.
Let us also fix $2\leq p_1, p_2<+\infty$ and $\delta >1+ \frac 1 {p_1}$.
Then there exists a positive constant $C$ such that for every $\lambda>0$,
$$
\mu\Big((v_0,v_1)\in {\mathcal H}^s\,:\,
\|\langle t \rangle ^{- \delta} S(t) (v_0,v_1)\|_{L^{p_1} (  \R_t ; L^{p_2}( \T^3))}> \lambda 
\Big)\leq C\exp\Bigl(- \frac {c\lambda^2} {\|(u_0, u_1)\|^2_{\mathcal{H}^0(\T^3)}}\Bigr)\,.
$$
\end{corollary}
Notice finally that using the Sobolev embeddings  $W^{\sigma, p}(\T^3) \subset L^\infty( \T^3), \sigma> \frac 3 p$, where the $W^{\sigma,p}$ norm is defined by 
$$ \|u \|_{W^{\sigma,p}( \T^3)}= \| (1- \Delta)^{\sigma/2} u \|_{L^p( \T^3)},$$
we obtain also
\begin{corollary} \label{corA.6}
Let us fix $s>0$ and $\mu\in {\mathcal M}^s$ Let $0<\sigma\leq s$ and let us
suppose that $\mu$ is induced via the map \eqref{eq.proba} from the couple $(u_0,u_1)\in {\mathcal H}^s$.
Let us also fix $2\leq p_1<+\infty$, $2\leq p_2 \leq + \infty$ and $\delta >1+ \frac 1 {p_1}$.
Then there exists a positive constant $C$ such that for every $\lambda>0$,
$$
\mu\Big((v_0,v_1)\in {\mathcal H}^s\,:\,
\|\langle t \rangle ^{- \delta} S(t) (v_0,v_1)\|_{L^{p_1} (  \R_t ; L^{p_2}( \T^3))}> \lambda 
\Big)\leq C\exp\Bigl(- \frac {c\lambda^2} {\|(u_0, u_1)\|^2_{\mathcal{H}^{\sigma}(\T^3)}}\Bigr)\,.
$$
\end{corollary}
\begin{remark}\label{random5} 
The same argument as in the proof of  Proposition~\ref{random1} shows that for every $2\leq p<+\infty$ and $s\geq 0$
there exist $C,c>0$ such that under the assumptions of  Proposition~\ref{random1} defining $\mu$,
for every $\lambda>0$ and every integer $N\geq 0$,
\begin{align}
\label{eq.random1}
\mu\Big((v_0,v_1)\in {\mathcal H}^s\,:\,\|\Pi_{N}v_0\|_{L^p (  \T^3)}> \lambda \Big) 
\leq C\exp\Bigl(- \frac {c\lambda^2} {\|(u_0, u_1)\|^2_{\mathcal{H}^0(\T^3)}}\Bigr),
\\ 
\label{eq.random2}
\mu\Big((v_0,v_1)\in {\mathcal H}^s\,:\,\|(v_0, v_1)\|_{\mathcal{H}^s (  \T^3)}> \lambda \Big) \leq 
C\exp\Bigl(- \frac {c\lambda^2} {\|(u_0, u_1)\|^2_{\mathcal{H}^s(\T^3)}}\Bigr).
\end{align}
\end{remark}
\subsection{Conditioned large deviation estimates}\label{sec.random2}
The purpose of this section is to deduce the following conditioned versions of our previous large deviation estimates:
\begin{proposition}\label{ochak}
Let $\mu\in {\mathcal M}^s$, $s\in(0,1)$ and suppose that the real random variable with distribution $\theta$, involved in the definition of $\mu$ is
symmetric. Then for $\delta> 1+ \frac{ 1} {p_1} $, $2\leq p_1<\infty$ and $2\leq p_2\leq\infty$ there exist positive constants
$c,C$ such that for every positive $\varepsilon,\lambda,\Lambda$ and $A$,
\begin{multline}\label{eq.cond}
\mu\otimes\mu\Big(
((v_0,v_1),(v'_0,v'_1))\in {\mathcal H}^s\times {\mathcal H}^s\,:\,\
\|\langle t \rangle ^{- \delta} S(t) (v_0- v'_0,v_1- v'_1)\|_{L^{p_1} (  \R_t ; L^{p_2}( \T^3))}> 
\lambda
\\
{\,\rm\,or\,}
\|\langle t \rangle ^{- \delta} S(t) (v_0+ v'_0,v_1+ v'_1)\|_{L^{p_1} (  \R_t ; L^{p_2}( \T^3))}> 
\Lambda
\Big\vert
 \|(v_0- v'_0,u_1- u'_1)\|_{\mathcal{H}^s(\T^3)}\leq \varepsilon
 \\
 {\rm\,and\,} \|(v_0+ v'_0,u_1+ u'_1)\|_{\mathcal{H}^s(\T^3)}\leq A  \Big) 
\leq 
C\Big(e^{-c\frac{\lambda^2}{\varepsilon^2}}+
e^{-c\frac{\Lambda^2}{A^2}}\Big).
\end{multline}
\end{proposition}
\begin{proof}
The proof of this result can be obtained by coming back to the original proof of Paley and Zygmund's of the $L^p$ boundedness of random series on the torus. However, we will follow a suggestion by J.P. Kahane and show that in fact, we can deduce it directly from the large deviation estimates proved in the previous section. The basic result which will allow this procedure, is the following lemma.
\begin{lemma}\label{lem.cond}
For $j=1,2$, let $E_j$ be two Banach spaces endowed with measures $\mu_j$. Let $f: E_1\times E_2\rightarrow\C$ and $g_1,g_2:E_2\rightarrow \C$ be three measurable functions. Then
\begin{multline*}
\mu_1\otimes\mu_2
\Big(
(x_1,x_2)\in E_1\times E_2\,:\, |f(x_1,x_2)|> \lambda\Big\vert|\, g_1(x_2)|\leq \varepsilon,\,\, |g_2(x_2)|\leq A
\Big)
\leq
\\
\sup_{x_2\in E_2, |g_1(x_2)|\leq \varepsilon, |g_2(x_2)|\leq A }
\mu_1(x_1\in E_1\,:\, |f(x_1,x_2)|>\lambda)\,,
 \end{multline*}
 where by $\sup$ we mean the essential supremum.
 \end{lemma} 
 \begin{proof}
 We may write
 \begin{multline}\label{dir}
 \int_{E_1} \chi(|f(x_1,x_2)|> \lambda)\chi(|g_1(x_2)|\leq \varepsilon)\chi(|g_2(x_2)|\leq A)
  d\mu_1(x_1)
 \\
 \leq
 \Big(
 \sup_{\genfrac{}{}{0pt}{}{X_2\in E_2, |g_1(X_2)|\leq \varepsilon}{|g_2(X_2)|\leq A}}
\mu_1(x_1\in E_1\,:\, |f(x_1,X_2)|> \lambda)
 \Big)\chi(|g_1(x_2)|\leq \varepsilon)\chi(|g_2(x_2)|\leq A)
 \end{multline}
 for a.e. $x_2\in E_2$.
 Here by $\chi(\cdot)$ we denote the characteristic function of the corresponding set.
 Now, we integrate the inequality \eqref{dir} over $x_2\in E_2$ with respect to $\mu_2$, to achieve the claimed bound.
 This completes the proof of Lemma~\ref{lem.cond}.
 \end{proof}
%
 We shall also use the following lemma.
 \begin{lemma}\label{ll2}
 Let $g_1$ and $g_2$ be two independent identically distributed real random variables with symmetric distribution. Then  $g_1\pm g_2$ have  symmetric distributions.
 Moreover if $h$ is a Bernoulli random variable independent of $g_1$ then $hg_1$ has the same distribution as $g_1$.
 \end{lemma} 
 The first part of the statement is straightforward if the distribution of $g_1$ and $g_2$ is absolutely continuous with respect to the Lebesgue measure
 (in the analysis of $g_1-g_2$ we do not need the symmetry assumption on $g_1$, $g_2$).
 In the general case one may invoke a duality and approximation argument.
 The second part of the lemma is straightforward.
 \par
 
 Let us now turn to the proof of Proposition~\ref{ochak}.
 Define
 $$
 {\mathcal E}\equiv \R\times \R^{\Z^3_{\star}}\times \R^{\Z^3_{\star}}\
 $$ 
 equipped with the natural Banach space structure coming from the $l^\infty$ norm. 
 We endow ${\mathcal E}$ with a probability measure $\mu_0$ defined via the map
 $$ 
 \omega\mapsto 
 \Big(
 k_{0}(\omega), \big(l_{n}(\omega)\big)_{n\in\Z^3_{\star}},
 \big(h_{n}(\omega)\big)_{n\in\Z^3_{\star}}
 \Big) ,
 $$
where $(k_0,l_{n}, h_n)$ is a system of independent Bernoulli variables. 
 
For $h=\big(x,(y_n)_{n\in\Z^3_{\star}},(z_n)_{n\in\Z^3_{\star}}\big)\in{\mathcal E}$ and 
$$
u(x)=a+\sum_{n\in\Z^3_{\star}}\Big(b_{n}\cos(n\cdot x)+c_{n}\sin(n\cdot x)\Big),
$$ 
we define the operation $\odot$ by
$$
h\odot u\equiv
ax+\sum_{n\in\Z^3_{\star}}\Big(b_{n} y_{n}\cos(n\cdot x)+c_{n}z_{n}\sin(n\cdot x)\Big).
$$
Let us first evaluate the quantity
\begin{multline}\label{ven_1}
\mu\otimes\mu\Big(
((v_0,v_1),(v'_0,v'_1))\in {\mathcal H}^s\times {\mathcal H}^s\,:\,\
\\
\|\langle t \rangle ^{- \delta} S(t) (v_0- v'_0,v_1- v'_1)\|_{L^{p_1} (  \R_t ; L^{p_2}( \T^3))}> 
\lambda
\Big\vert
\\
 \|(v_0- v'_0,v_1- v'_1)\|_{\mathcal{H}^s(\T^3)}\leq \varepsilon
 {\rm\,and\,} \|(v_0+ v'_0,v_1+ v'_1)\|_{\mathcal{H}^s(\T^3)}\leq A  \Big).
\end{multline}
Observe that, thanks to Lemma~\ref{ll2}, \eqref{ven_1} equals
\begin{multline}\label{ven_2}
\mu\otimes\mu\otimes\mu_0\otimes\mu_0\Big(
((v_0,v_1),(v'_0,v'_1), (h_0,h_1))\in {\mathcal H}^s\times {\mathcal H}^s\times {\mathcal E}\times {\mathcal E}\,:\,\
\\
\|\langle t \rangle ^{- \delta} S(t) (h_0\odot (v_0-v'_0),h_1\odot (v_1- v'_1))\|_{L^{p_1} (  \R_t ; L^{p_2}( \T^3))}> 
\lambda
\Big\vert
\\
 \|(h_0\odot(v_0- v'_0),h_1\odot(v_1- v'_1))\|_{\mathcal{H}^s(\T^3)}\leq \varepsilon
 {\rm\,and\,} \|(h_0\odot(v_0+ v'_0),h_1\odot(v_1+ v'_1))\|_{\mathcal{H}^s(\T^3)}\leq A  \Big). 
\end{multline}
Since the $H^s(\T^3)$ norm of a function $f$ depends only on the absolute value of its Fourier coefficients, we deduce that \eqref{ven_2} equals
\begin{multline}\label{ven_3}
\mu\otimes\mu\otimes\mu_0\otimes\mu_0\Big(
((v_0,v_1),(v'_0,v'_1), (h_0,h_1))\in {\mathcal H}^s\times {\mathcal H}^s\times {\mathcal E}\times {\mathcal E}\,:\,\
\\
\|\langle t \rangle ^{- \delta} S(t) (h_0\odot (v_0-v'_0),h_1\odot (v_1- v'_1))\|_{L^{p_1} (  \R_t ; L^{p_2}( \T^3))}> 
\lambda
\Big\vert
\\
 \|(v_0- v'_0,v_1- v'_1)\|_{\mathcal{H}^s(\T^3)}\leq \varepsilon
 {\rm\,and\,} \|(v_0+ v'_0,v_1+ v'_1)\|_{\mathcal{H}^s(\T^3)}\leq A  \Big) 
\end{multline}
We now apply Lemma~\ref{lem.cond}  with $\mu_1=\mu_0\otimes\mu_0$ and $\mu_2=\mu\otimes\mu$
to get that \eqref{ven_3} is bounded by
\begin{multline}\label{ven_4}
\sup_{ \|(v_0- v'_0,v_1- v'_1)\|_{\mathcal{H}^s(\T^3)}\leq \varepsilon}
\mu_0\otimes\mu_0\Big(
 (h_0,h_1)\in  {\mathcal E}\times {\mathcal E}\,:\,\
\\
\|\langle t \rangle ^{- \delta} S(t) (h_0\odot (v_0-v'_0),h_1\odot (v_1- v'_1))\|_{L^{p_1} (  \R_t ; L^{p_2}( \T^3))}> 
\lambda \Big) 
\end{multline}
We now apply Corollary~\ref{corA.6} (with Bernoulli variables) to obtain that \eqref{ven_1} is bounded by
$
C\exp(-c\frac{\lambda^2}{\varepsilon^2}).
$
A very similar argument gives that
\begin{multline*}
\mu\otimes\mu\Big(
((v_0,v_1),(v'_0,v'_1))\in {\mathcal H}^s\times {\mathcal H}^s\,:\,\
\\
\|\langle t \rangle ^{- \delta} S(t) (v_0+ v'_0,v_1+ v'_1)\|_{L^{p_1} (  \R_t ; L^{p_2}( \T^3))}> 
\Lambda
\Big\vert
\\
 \|(v_0- v'_0,v_1- v'_1)\|_{\mathcal{H}^s(\T^3)}\leq \varepsilon
 {\rm\,and\,} \|(v_0+ v'_0,v_1+ v'_1)\|_{\mathcal{H}^s(\T^3)}\leq A  \Big) 
\end{multline*}
is bounded by 
$
C\exp(-c\frac{\Lambda^2}{A^2}).
$
This completes the proof of Proposition~\ref{ochak}.
 \end{proof}
\section{Properties of the measures $\mu_{(u_0,u_1)}$}
Via the choice of coordinates induced by the decomposition~\eqref{coord}
 $$(u_0, u_1) \in \mathcal{H}^s\mapsto (a_0,(b_{n,0}, c_{n,0})_{n\in \mathbb{Z}^3_*}, a_1, (b_{n,1}, c_{n,1})_{n\in \mathbb{Z}^3_*})\in \bigl(\mathbb{R}\times \mathbb{R}^{\mathbb{Z}^3_*}\times \mathbb{R}^{\mathbb{Z}^3_*}\bigr)^2$$
the measure $\mu_{(u_0,u_1)}$ can be seen as an infinite tensor product of probability measures on
 $$ 
 \bigl(\mathbb{R}\times \mathbb{R}^{\mathbb{Z}^3_*}\times \mathbb{R}^{\mathbb{Z}^3_*}\bigr)^2\,,
 $$
 $$ 
 \mu \sim \mu_{0,0}\otimes_{n\in \mathbb{Z}^3_*} \mu_{n,0,b} \otimes_{n\in
 \mathbb{Z}^3_* } \mu_{n,0,c} \otimes \mu_{0,1}\otimes_{n\in \mathbb{Z}^3_*} \mu_{n,1,b} \otimes_{n\in \mathbb{Z}^3_*} \mu_{n,1,c} 
 $$ 
 where 
 $$
 \mu_{0,0},\mu_{n,0,b}, \mu_{n,0,c}, \mu_{0,1}, \mu_{n,1,b}, \mu_{n,1,c}
 $$
 are the distributions of the random variables
 $$
 a_0\alpha_0,  b_{n,0}\beta_{n,0}, c_{n,0}\gamma_{n,0}, a_1\alpha_1, b_{n,1}\beta_{n,1}, c_{n,1}\gamma_{n,1}
 $$
 respectively.  As a consequence, we will be able to apply the following result by Kakutani~\cite{Ka}.
 \begin{theorem}\label{kakutani}
 Consider the infinite tensor products of probability measures on $\mathbb{R}^{\mathbb{N}}$
 $$ \mu_i= \bigotimes _{n\in \mathbb{N}} \mu_{n,i}, \qquad i=1,2.
 $$
 Then the measures $\mu_1$ and $\mu_2$ on $\mathbb{R}^{\mathbb{N}}$ endowed with its cylindrical Borel $\sigma$-algebra are absolutely continuous with respect each other, $\mu_1 \ll \mu_2$, and $\mu_2 \ll \mu_1$, if and only if the following holds:
 \begin{enumerate}
 \item The measures $\mu_{n,1}$ and $\mu_{n,2}$ are for each $n$ absolutely continuous with respect to each other: there exists two  functions $g_n \in L^1( \mathbb{R}, d\mu_{n,2})$, $k_n \in L^1( \mathbb{R}, d\mu_{n,1}) $ such that 
 $$ d\mu_{n,1} = g_n d\mu_{n,2}, \qquad d\mu_{n,2} = k_n d\mu_{n,1}
 $$
 \item The functions $g_n$ are such that the infinite product 
 \begin{equation}
 \label{prodinf}
  \prod_{n\in \mathbb{N}} \int_{\mathbb{R} }g_n^{1/2} d\mu_{n,2} =  \prod_{n\in \mathbb{N}} \int_{\mathbb{R} } \sqrt{ d\mu_{n,1} }\sqrt{ d\mu_{n,2}}
  \end{equation}
 is convergent (i.e. positive).
 \end{enumerate}
 Furthermore, if any of the condition above is not satisfied (i.e. if the two measures $\mu_1$ and $\mu_2$ are not absolutly continuous with respect to each other), then the two measures are mutually singular: there exists a set $A\subset \mathbb{R}^{\mathbb{N}}$ such that 
 $$ \mu_1 (A) =1, \qquad \mu_2 (A) =0$$
 \end{theorem}
 Theorem~\ref{kakutani} implies the following statement concerning the measures we studied in this paper in the context of the cubic wave equation \eqref{NLW}.
 \begin{proposition}
  Assume that the random variables  $(\alpha_{j}(\omega),\beta_{n,j}(\omega),\gamma_{n,j}(\omega))$, $j=0,1$,  $n\in \mathbb{Z}^3_*$,  used to obtain the randomisation as explained in the introduction are independent centered gaussian random variables.
Let 
 $$
 u_{j}(x)= a_{0,j}+ \sum_{n\in \mathbb{Z}^3_*} \Big(b_{n,j} \cos (n\cdot x) + c_{n,0} \sin(n\cdot x)\Big), \qquad j=0,1,
 $$
 $$
 \widetilde{u}_{j}(x)= \widetilde{a}_{0,j}+ \sum_{n\in \mathbb{Z}^3_*} \Big(\widetilde{b}_{n,j} \cos (n\cdot x) + \widetilde{c}_{n,0} \sin(n\cdot x)\Big), \qquad j=0,1.
 $$
 Then the measures $\mu_{(u_0, u_1)}$ and $\mu_{(\widetilde{u}_0, \widetilde{u}_1)}$ are absolutely continuous with respect to each other if 
 and only if neither of the coefficients ($a,b,c, \widetilde{a}, \widetilde{b}, \widetilde{c}$) 
 above vanishes (or then they must vanish simultaneously, i.e. if $a_{0,j}=0$, then $\widetilde{a} _{0,j}=0$, etc...and 
 $$
 \sum_{j=0}^1\Bigl(\Bigl| \frac{ \widetilde{a}_{0,j}}{ a_{0,j}}\Bigr| -1 \Bigr)^2+ 
 \sum_{n\in \mathbb{Z}^3_*} \Bigl( \Bigr|\frac{ \widetilde{b}_{n,j}}{ b_{n,j}}\bigr|-1 \Bigr) ^2 +
  \Bigl( \Bigl|\frac{ \widetilde{c}_{n,j}}{ c_{n,j}} \Bigr|-1 \Bigr)^2 < + \infty.
 $$
 Furthermore, of this condition is not satisfied, then the two measures $\mu_{(u_0, u_1)}$ and $\mu_{(\widetilde{u}_0, \widetilde{u}_1)}$ are mutually singular.
 \end{proposition}
 \begin{proof}Indeed, if $g$ is a normalized gaussian random variable, the random variable $\alpha g$ is a Gaussian random variable centered and variance $\alpha^2$, 
 and eliminating the trivial contributions when the coefficients vanish simultaneously, the result amounts to proving that if $\mu_i= \otimes_{n\in \N}\mu_{n,i}$, with $\mu_{n,i}$ Gaussian distribution of variance $x_{n,i}^2$,  then the measures $\mu_1$ and $\mu_2$ are absolutely continuous with respect to each other if and only if 
 $$ \sum_{n} \Bigl| \frac {x_{n,1}} { x_{n,2}} -1 \Bigr| ^2 <+ \infty$$
 in this case, we have 
 $$ d\mu_{n,i}= \frac{ 1}{ x_{n,i} \sqrt{2\pi}} e^{- \frac{ t^2}{ 2x_{n,i}^2}} dt$$
 and 
 $$ g_n= \frac{ x_{n,2}} { x_{n,1}} e^{\frac{ t^2} { 2x_{n,2}^2}-\frac{ t^2}{ 2x_{n,1}^2}}\,.
 $$
 Consequently, 
 \begin{equation}
 \int_{\R} g_n^{1/2} d\mu_{n,2}= 
 \int_{\R} \frac{ 1} { \sqrt{2\pi x_{n,1}x_{n,2}}} e^{-\frac{ t^2} { x_{n,2}^2}-\frac{ t^2}{ x_{n,1}^2}} dt\\
 =\Bigl(\frac{ 2 x_{n,1} x_{n,2} } { x_{n,1}^2+ x_{n,2}^2} \Bigr) ^{\frac 1 2}=\Bigl(\frac{ \frac {x_{n,1}} {x_{n,2}} + \frac {x_{n,2}} {x_{n,1}}} { 2} \Bigr)^{- \frac 1 2}  
   \end{equation}
  and we deduce that if the infinite product~\eqref{prodinf} is convergent then necessarily the quotients $\frac {x_{n,1}} {x_{n,2}} $ tend to $1$.  Now, writing $\frac {x_{n,2}} { x_{n,1}} =1+ \varepsilon_n$, we have 
  $$  \Bigl(\frac{ 2 x_{n,1} x_{n,2} } { x_{n,1}^2+ x_{n,2}^2} \Bigr) ^{\frac 1 2}=  1- \frac 1 4  \varepsilon_n^2 + \mathcal{O} ( \varepsilon_n ^3)\,.
  $$ 
  Finally, by taking the logarithm, we conclude that the infinite product~\eqref{prodinf} is convergent if and only if 
  $$ 
  \sum_n \varepsilon_n ^2 <+ \infty.
  $$ 
 \end{proof}


\end{document}